\documentclass[final]{siamart190516}

\usepackage{amssymb,amsmath,amsfonts,mathtools,enumerate,color}

\usepackage[toc,page,title,titletoc,header]{appendix}
\usepackage{graphicx,subfig}

\usepackage{paralist}

\usepackage{booktabs}

\usepackage{tikz}
\usetikzlibrary{arrows,positioning,shapes.geometric}

\newtheorem{Theorem}{Theorem}[section]
\newtheorem{Lemma}[Theorem]{Lemma}
\newtheorem{Proposition}[Theorem]{Proposition}
\newtheorem{Assumption}{H.\!\!}

\theoremstyle{definition}
\newtheorem{Definition}{Definition}[section]

\newtheorem{Example}{Example}[section]

\theoremstyle{remark}
\newtheorem{Remark}{Remark}[section]

  \def\nb{\nonumber}
\def \Vh0{\stackrel{\circ}{V}_h} \def\to{\rightarrow}
   
\def\Om{\Omega}  \def\om{\omega} \def\I{ {\rm (I) } }
 
\newcommand{\q}{\quad}

\def\l{\label}  \def\f{\frac}  \def\fa{\forall}
\def\b{\beta}  \def\a{\alpha} 
 
\def\eps{\varepsilon}

 \def\t{\times}

\def\u{{\bf u}}

\def\cA{\mathcal{A}}
\def\cB{\mathcal{B}}

\def\cF{\mathcal{F}}

\def\cI{\mathcal{I}}

\def\cK{\mathcal{K}}
\def\cL{\mathcal{L}}
\def\cM{\mathcal{M}}

\def\cO{\mathcal{O}}

\def\cS{\mathcal{S}}

\def\cZ{\mathcal{Z}}

\def\bA{{\textbf{A}}}

\def\I{\mathbb{I}}
\def\N{{\mathbb{N}}}
\def\bP{\mathbb{P}}

\def\R{{\mathbb R}}
\def\bS{\mathbb{S}}
\def\Z{{\mathbb{Z}}}

\newcommand{\ex}{\mathbb{E}}

\DeclareMathOperator*{\argmax}{arg\,max}
\DeclareMathOperator*{\argmin}{arg\,min}

\def\bb{\begin{equation}} \def\ee{\end{equation}}
\def\bn{\begin{enumerate}} \def\en{\end{enumerate}}

\begin{document}

\title{Error estimates of penalty schemes for quasi-variational inequalities arising from  impulse control problems}
\author{
Christoph Reisinger\thanks{Mathematical Institute, University of Oxford,  Oxford OX2 6GG, UK (\email{christoph.reisinger@maths.ox.ac.uk}, \email{yufei.zhang@maths.ox.ac.uk}).}
\and
Yufei Zhang\footnotemark[1]
}

\maketitle

\begin{abstract}
This paper proposes penalty schemes for a class of weakly coupled systems of Hamilton-Jacobi-Bellman quasi-variational inequalities (HJBQVIs) arising  from stochastic hybrid control problems of regime-switching models with both continuous and impulse controls. We show that the solutions of the penalized equations converge monotonically to those of the HJBQVIs. We  further
establish that the schemes are half-order accurate for HJBQVIs with Lipschitz  coefficients, and  first-order accurate for equations with more regular coefficients. Moreover, we   construct   the  action regions and optimal impulse controls based on the error estimates and the penalized solutions. 
The penalty schemes and convergence results are then extended to HJBQVIs with possibly negative impulse costs. 
We also demonstrate the convergence of monotone discretizations of the penalized equations, and establish that policy iteration applied to the discrete equation is monotonically convergent with an arbitrary initial guess in an infinite dimensional setting. 
Numerical examples for infinite-horizon optimal switching problems are presented  to illustrate  the effectiveness of the penalty schemes over the conventional direct control scheme.

\end{abstract}



\section{Introduction}\l{sec:intro}

In this paper we study penalty schemes and their convergence  for the following weakly coupled system of degenerate Hamilton-Jacobi-Bellman quasi-variational inequalities (HJBQVIs): for all $i\in\cI:=\{1,\dots,M\}$,
\begin{align}
\l{eq:general}
\max\Big\{
\sup_{\alpha\in\cA_i}\mathcal{L}^{\alpha}_i(x,u(x),Du_i(x),D^2u_i(x)), \    
(u_i-\mathcal{M}_i u)(x)\Big\}=0,   \q x\in \R^d, 
\end{align}
where  $u=(u_i)_{i\in\cI}$ denotes the unknown solution, $(\cL^\a_i)_{i\in\cI}$ is a family of second order differential operators,
and $\cM_i$ is an intervention operator of the following form:
\bb
\label{eq:M_impulse_intro}
(\mathcal{M}_i u)(x) = \min_{z \in Z_i(x) } \{ u_i(\Gamma_i(x,z)) + K_i(x,z) \}.
\ee

The above system extends the classical scalar HJBQVIs, and arises naturally from hybrid control problems of regime-switching models with both  continuous and impulse controls (see e.g.~\cite{bensoussan1997,wei2010,yin2010,sotomayor2013,tan2018}).
For instance, let $\a$ be a c\`{a}dl\`{a}g adapted stochastic control process, 
and let $\gamma=(\tau_1,\xi_1;\tau_2,\xi_2;\ldots)$ be an impulse control strategy consisting of a sequence of impulse times $0=\tau_0\le \tau_1\le \tau_2\le \ldots,$ and  adapted impluse controls $(\xi_1,\xi_2,\ldots)$. Between impulse times, we assume the state process  $X$ follows a controlled regime-switching process defined as follows: $X_0=x\in \R^d$, $I_0=i\in \cI$, and for all $k\in \N\cup\{0\}$,
$$
dX_t=b(\a_t, I_t,X_t) \, dt+\sigma(\a_t, I_t,X_t) \, dW_t, \q \tau_k<t<\tau_{k+1},
$$
where $W$ is a standard Brownian motion, and $I$ is a continuous-time Markov chain with values in the finite set $\cI$,  
{
which represents the uncertainty in the environment and randomly switches among $M=|\mathcal{I}|$ states, governed by a controlled Markov transition matrix $(d_{ij}^{\a_t}(X_t))_{i,j \in \mathcal{I}}$.}
At an impulse time $\tau_k$, the impulse control $\xi_k$ is applied and instantaneously changes the state  into
$X_{\tau_{k}}=\Gamma(I_{\tau_{k}^-},X_{\tau_{k}^-},\xi_{k})$. 
The aim is to minimize the expected cost over all admissible strategies $(\a,\gamma)$ by considering the following value function:
\bb\l{eq:u_i}
u_i(x)\coloneqq \inf_{\a,\gamma}\ex\bigg[\int_0^\infty \ell(\a_t,I_t,X_t)e^{-c(I_t,X_t)t}\,dt+\sum_{k=1}^\infty K(I_{\tau_{k}^-},X_{\tau_{k}^-},\xi_{k}) e^{-c(I_{\tau_{k}^-},X_{\tau_{k}^-})\tau_k}\bigg]
\ee
for each $x\in \R^d$ and $i\in\cI$, where $\ell$ and $K$ are the running cost and the impulse cost, respectively.

Such hybrid control problems appear in mathematical finance, such as in the following applications: portfolio optimization with transaction costs \cite{korn1999,oksendal2005,azimzadeh2016weakly},  control of exchange rates \cite{korn1999,oksendal2005,davis2010,azimzadeh2016weakly}, credit securitization \cite{seydel2009}, inventory control and  dividend control \cite{korn1999,bai2010}.
It is well-known that  under suitable assumptions, the value functions $(u_i)_{i\in \cI}$ in \eqref{eq:u_i} can be characterized by the viscosity solution to \eqref{eq:general} (see e.g. \cite{wei2010,tan2018}).
{Note that due to the random switching process $I$,  each operator $\cL^\a_i$  involves all components of the solution $u$, which leads us to a weakly coupled system of HJBQVIs (see e.g.~\cite{yin2010,briani2012}).}

As the solution to \eqref{eq:general} is in general not known analytically, several classes of numerical schemes have been proposed to solve such nonlinear equations. By writing the obstacle term $u_i-\cM_i u$ as $\max_{z\in Z_i(x)} [u_i-u_i(\Gamma_i(x,z)) - K_i(x,z)]$, one can extend the ``direct control" scheme of HJB equations to solve \eqref{eq:general}, which discretizes the operators in \eqref{eq:general} and attempts to solve the resulting nonlinear discrete equations using policy iteration \cite{chancelier2007,azimzadeh2016weakly}. However, due to the non-strict monotonicity of the term $u_i-\cM_i u$, such a scheme in general requires a very accurate initial guess for the policy iteration to converge. In fact, as we shall show in Remark \ref{rmk:direct}, even for some simple intervention operators, 
policy iteration in the direct control scheme may  not be well-defined for an arbitrary initial guess due to the possible singularity of the matrix iterates. 

An alternative approach to solving  \eqref{eq:general}, referred to as  iterated optimal stopping,  approximates the QVI by  a sequence of HJB variational-inequalities (see \eqref{eq:iter_0}--\eqref{eq:iter_n}),
which can subsequently be solved by the direct control scheme \cite{oksendal2005,seydel2009}. However, since  this approach  can be equivalently formulated as a fixed point algorithm for the QVI, one can show that this approach in general suffers from slow convergence (i.e.\ rate close to 1), especially for small impulse costs \cite{reisinger2018qvi}.

In this work, we shall extend the penalty schemes in \cite{kharroubi2010,azimazadeh2018} for scalar equations (i.e.~$M=1$) to systems of HJBQVIs, and construct the solution of \eqref{eq:general} from a sequence of penalized equations.
  The major advantage of the penalty approximation is that one can easily construct convergent monotone discretizations of the penalized equation with a fixed penalty parameter, and policy iteration applied to the discrete equation is monotonically convergent with any initial guess (see Section \ref{sec:discrete}). Moreover, the Lagrange multipliers of the penalized equations enjoy better regularity than those of the unpenalized QVI \eqref{eq:general}. 
{It is observed empirically that this improved regularity leads to   mesh-independent behaviour of policy iteration for solving the penalized equations, i.e., the number of iterations for solving the discrete problem remains bounded as the mesh size tends to zero, in contrast to the direct scheme (see Figure \ref{fig:efficiency} in Section \ref{sec:num}, see also \cite{reisinger2012,hintermuller2007}).}
  
  All these  appealing features  motivate us to design  efficient penalty schemes for solving systems of HJBQVIs with general intervention operators. We further establish that as the penalty parameter $\rho$ tends to infinity, the solution of the penalized equation converges monotonically from above to the solution of the HJBQVI.
We shall also construct novel convergent approximations of the  action regions and optimal impulse control strategies based on the penalized solutions.

Another major contribution of this work is the convergence rate of such penalty approximations for degenerate HJBQVIs, which is novel even in the scalar case (i.e.~$M=1$). Although the convergence of penalty schemes for QVIs has been proved in various works (e.g.~\cite{lions1982,kharroubi2010,azimazadeh2018}), to the best of our  knowledge, there is no published work on  the accuracy of the penalty approximation with a given penalty parameter (except for, \cite{reisinger2018qvi} where the penalty error for discrete QVIs has been analyzed). This is not only   important for  the choice of penalty parameters and 
the practical implementation of penalty schemes, but is also  crucial for the construction of action regions  (see Remark \ref{rmk:region}) and optimal impulse control strategies. In this work, we shall close the gap by giving a rigorous analysis of the penalty errors.

Let us briefly comment on the two main difficulties encountered in deriving the error estimates.  
In contrast to the results for  finite-dimensional (discretized) QVIs in \cite{reisinger2018qvi}, the convergence rate of penalty approximations for HJBQVIs depends on the regularity of the solution.
Since in this work we focus on degenerate HJBQVIs, including the fully degenerate case where $\cL^\a_i$ reduces to a first-order differential operator, 
the solution of \eqref{eq:general} is typically not differentiable due to the lack of regularization from the Laplacian  operator. Therefore, we need to obtain suitable regularity of the solution to weakly coupled systems based on  viscosity solution theory \cite{crandall1992}.

Moreover, the non-diagonal dominance of the  obstacle term $u_i-\cM_i u$ poses a significant challenge for estimating the penalization errors. In fact, a crucial step in estimating the penalty error for HJB variational-inequalities is to show that there exists a constant $C$, depending  on the regularity of the obstacle, such that for any $\rho>0$, if $u^\rho$ solves the penalized equation with the  parameter $\rho$, then $u^\rho-C/\rho$ satisfies the constraint of the variational inequality (see e.g.~\cite{jakobsen2006,witte2012}). However, this is in general false for the QVIs since the term $u_i-\cM_i u$  remains invariant under any vertical shift of the solutions.  

We shall overcome the above difficulty by {combining the ideas of \cite{bonnans2007,jakobsen2006} and precise regularity estimates (i.e., Lipschitz continuity and semiconcavity) of solutions to HJB variational inequalities.}
In particular, we shall construct a family of auxiliary approximations for our error analysis via  iterated optimal stopping. 
This reduces the problem to estimating the solution regularity and  penalty errors for a sequence of obstacle problems.
We shall derive a  more precise estimate for  the semiconcavity constant of the solution to HJB variational inequalities with respect to the obstacle term than those in prior works (see the discussion above Proposition \ref{prop:semiconcave_ob}). This is crucial for us to be able to conclude that  the penalty approximation is  half-order accurate for HJBQVIs with Lipschitz coefficients, and first-order accurate for equations with more regular coefficients  (see Theorems \ref{thm:error_impulse} and \ref{thm:error_switching}). 
These convergence rates of penalty schemes for HJBQVIs are optimal in the sense that they are of the same order (up to  logarithmic terms) as those for conventional HJB variational inequalities.

We further extend the penalty scheme and its error estimate to a class of HJBQVIs with possibly \textit{negative impulse costs}.
Note that signed costs are not only of mathematical interest, but are also important to model the situation where the controller can obtain a positive impulse benefit, for example, receive financial support for investing  in renewable  energy production (see \cite{pham2009,lundstrom2014}).
In this setting, we deduce error estimates for a different type of penalty schemes, which apply the penalty to each  impulse control strategy, instead of  the pointwise maximum over all impulse control strategies (Remark \ref{rmk:pointwise}). These convergence results rely on  a novel construction of a strict subsolution to HJBQVIs with  general switching costs, for which we  
impose less restrictive conditions on the switching costs than those given in the literature (see the discussion after (H.\ref{assum:sc}) for details).

Finally, we  would like to point out a control-theoretic interpretation of our penalty schemes. As observed in \cite{liang2015,liang2016}, the viscosity solution of the penalized equation with parameter $\rho$ can be identified as the value function of a  hybrid control problem where the controller is only allowed to
perform impulse controls at a sequence of Poisson arrival times with intensity $\rho$, instead of any stopping times.
Our error estimates give a convergence rate   of these  hybrid control problems with random intervention times in terms of  the intensity $\rho$, which is of independent interest.

We organize this paper as follows. Section \ref{sec:qvi_positive} states the main assumptions and recalls basic results for the system of  HJBQVIs with positive impulse costs. In Section \ref{sec:penalty} we shall propose a penalty approximation to the HJBQVIs  and establish its monotone convergence. Then by exploiting the regularization  introduced in Section \ref{sec:regularize}, we  estimate the convergence rates of the penalty schemes in Section \ref{sec:rate}, and construct convergent approximations to action regions and optimal impulse controls in Section \ref{sec:region_ctrl}.
We extend the convergence results to HJBQVIs with signed costs in Section \ref{sec:switching}, and discuss the monotone convergence of policy iteration in Section \ref{sec:discrete}. 
Numerical examples for infinite-horizon optimal switching problems are presented in Section \ref{sec:num} to illustrate  the effectiveness of the penalty schemes.
Appendix \ref{appendix} is devoted to the proofs of some technical results.

\section{HJBQVIs with positive costs}\l{sec:qvi_positive}

In this section, we introduce the system of HJBQVIs of our interest, state the main assumptions on its coefficients, and recall the appropriate notion of solutions. We start with some useful notation which is needed frequently throughout this work.

For a function $\phi:\R^d\to \R$, we define the following (semi-)norms:
$$
|\phi|_0=\sup_{x\in \R^d}|\phi(x)|,\q [\phi]_1=\sup_{x,y\in \R^d}\f{|\phi(x)-\phi(y)|}{|x-y|}, \q |\phi|_1=|\phi|_0+[\phi]_1.
$$
As usual, we denote by $C^{0}(\R^d)$ (resp.~$C^n(\R^d)$) the space of bounded continuous functions (resp.~$n$-times differentiable functions) in $\R^d$, and by 
${C}^{0}_{1}(\R^d)$  the subset of functions in $C^0(\R^d)$ with finite $|\cdot|_1$ norm. Finally, we shall  denote by $\bS^d$ the set of $d\t d$ symmetric matrices, and by $X\ge Y$ in $\bS^d$ the fact that $X-Y$ is positive semi-definite. 

We shall consider the following weakly coupled system: for each $i\in\cI:=\{1,\dots,M\}$, and $x\in \R^d$, 
\begin{align}
\begin{split}
\l{eq:qvi_impulse}
&F_i(x,u,Du_i,D^2u_i)\\
&
\coloneqq \max\Big\{
\sup_{\alpha\in\cA_i}\mathcal{L}^{\alpha}_i(x,u(x),Du_i(x),D^2u_i(x)), \    
(u_i-\mathcal{M}_i u)(x)\Big\}=0,  
\end{split}
\end{align}
where $u=(u_i)_{i\in\cI}$, $\cL^\a_i: \R^d\t \R^M\t \R^d\t \bS^d\to \R$ is the following linear operator:
\bb
\label{eq:defLi}
\mathcal{L}^{\alpha}_i(x,s,p,X)=-\mathrm{tr}[a^{\alpha}_i(x) X] -
b^{\alpha}_i(x) p +c^\a_i(x)s_i-\ell_i^\a (x)-\sum_{j\in \cI^{-i}}d_{ij}^\a(x)s_j,
\ee
with $\cI^{-i}\coloneqq \{j\in \cI\mid j\not=i\}$, $\cM_i$ is the intervention operator \eqref{eq:M_impulse_intro}, i.e.,
\bb
\label{eq:M_impulse}
(\mathcal{M}_i u)(x) = \min_{z \in Z_i(x) } \{ u_i(\Gamma_i(x,z)) + K_i(x,z) \}.
\ee

Before introducing the assumptions on the coefficients, let us recall the concept of semiconcavity of a continuous function \cite{crandall1992,bardi1997}, which is crucial for  the subsequent convergence analysis. 
\begin{Definition}[Semiconcavity] \label{def:semiconcave}
A continuous function $\phi$
is semiconcave around $x\in \R^d$ with constant $C\ge 0$, if it holds  that
\begin{align}\l{eq:def_semiconcave}
\phi(x+h)-2\phi(x)+\phi(x-h)&\le C|h|^2, \q \textnormal{for all sufficiently small $h\in \R^d$.}
\end{align}
We say a continuous function $\phi$ is semiconcave with  constant $C\ge0$  if \eqref{eq:def_semiconcave} holds for all $x\in \R^d$. For any given semiconcave function $\phi$, we shall  denote by $[\phi]_{2,+}$ its semiconcavity constant, i.e.,
 $$[\phi]_{2,+}\coloneqq \inf\{C\ge 0 \mid u(x+h)-2u(x)+u(x-h)\le C|h|^2,\, x, h\in \R^d\}.$$
\end{Definition}

A concave function is clearly semiconcave. Moreover, a $C^1$ function with locally Lipschitz gradient is semiconcave \cite{bardi1997}.

We now list the main assumptions on the coefficients.

\begin{Assumption}\l{assum:regularity}
For any $i,j\in \cI$, $\cA_i$ is a nonempty compact set,  $a^{\alpha}_i=\f{1}{2}\sigma^{\alpha}_i{\sigma^{\alpha}_i}^T$
for some  $\sigma^{\alpha}_i\in \R^{d\t d'}$, and $\sigma^\a_i$, $b^\a_i$, $\ell^\a_i$, $c^\a_i$, $d^\a_{ij}$ are continuous functions. Moreover,
there exist constants $C$ and $\lambda_0$ such that it holds for any $ j\not =i$, $\a\in \cA_i$ that
\begin{align}
&|\sigma^{\alpha}_i|_1+|b^{\alpha}_i|_1 
+|\ell_i^\a|_1+|c_i^\a|_1+|d_{ij}^\a|_1\leq C, \l{eq:coeff_regularity}\\
&d^\a_{ij}\ge 0, \q c^\a_i-\sum_{j\in \cI^{-i}}d^\a_{ij}\ge \lambda_0>0.\l{eq:coeff_mono}
\end{align}
\end{Assumption}

\begin{Assumption}\l{assum:impulse}
For any $i\in \cI$ and $x\in \R^d$, $Z_i(x)$ is a nonempty compact set in a metric space $(\boldsymbol{Z},d_{\boldsymbol{Z}})$, $\Gamma_i$ and $K_i$ are continuous functions, and the mapping $x\to Z_i(x)$ is continuous in the Hausdorff metric. Moreover, 
there exists a constant $\kappa_0$ such that  for all  $i \in \cI$, $x\in \R^d$ and $z\in Z_i(x)$, we have $K_i(x,z)\ge \kappa_0>0$.
\end{Assumption}

The condition \eqref{eq:coeff_regularity} in (H.\ref{assum:regularity}) is the standard regularity assumption for the coefficients in viscosity solution theory, while  \eqref{eq:coeff_mono} in (H.\ref{assum:regularity}) implies the monotonicity of the HJB equations. 
The condition (H.\ref{assum:impulse}) on the intervention operator is the same as that in \cite{seydel2009,azimazadeh2018}, which ensures the well-posedness of \eqref{eq:qvi_impulse} in the class of bounded continuous functions.
As we shall show in Section \ref{sec:penalty}, they are sufficient for the monotone convergence of the penalty approximation, even for non-convex/non-concave systems involving Isaacs' equations.

The following additional assumptions are necessary to derive  the regularity of the value functions and quantify the error estimates of the penalty schemes.

\begin{Assumption}\l{assum:lambda_lip}
The constant $\lambda_0$ in (H.\ref{assum:regularity}) satisfies $\lambda_0>\sup_{\a,i}([\sigma^{\alpha}_i]^2_1+[b^\a]_1)$.
\end{Assumption}

\begin{Assumption}\l{assum:regularity_concave}
There exists a constant $C> 0$ such that for any $i,j\in \cI$, $\a\in \cA_i$, we have $\sigma^\a_i, b^\a_i, c^\a_i, d^\a_{ij}\in C^1(\R^d)$ satisfying the estimate
$$
|D \sigma^{\alpha}_i|_1+|D b^{\alpha}_i|_1 +|D c_i^\a|_1+|D d_{ij}^\a|_1\leq C,
$$
and $\ell^\a_i$ is semiconcave with  constant $C$ in $\R^d$.
\end{Assumption}

\begin{Assumption}\l{assum:M_lip}
For any $i\in\cI$, the  operator $\cM_i$ preserves Lipschitz functions, i.e., 
there exists a constant $C> 0$ such that for any $u\in C^0_1(\R^d)$, $\cM_i u$ is  Lipschitz continuous with a constant satisfying $[\cM_i u]_1\le [u]_1+C$.
\end{Assumption}

\begin{Assumption}\l{assum:M_concave}
For any $i\in\cI$, the  operator $\cM_i$ preserves semiconcave functions, i.e., 
there exists a constant $C> 0$ such that for any given bounded semiconcave function $u$, $\cM_i u$ is  semiconcave with a constant satisfying $[\cM_i u]_{2,+}\le [u]_{2,+}+C$.
\end{Assumption}

Let us briefly discuss  the importance of the above assumptions. 
The condition (H.\ref{assum:lambda_lip}) is the standard assumption for the Lipschitz continuity of the value functions (see \cite{lions1982}), while 
(H.\ref{assum:regularity_concave}) will be used to establish the semiconcavity of the solutions in Section \ref{sec:regularize}, which is the maximal regularity that one can expect for the solutions of degenerate HJB equations (see e.g.~\cite{bardi1997}).

Conditions  (H.\ref{assum:M_lip}) and (H.\ref{assum:M_concave}) are certain structural assumptions for the intervention operator $\cM_i$, which play an essential role in our error estimates. In general  these conditions need to be verified in a problem dependent way, as demonstrated in the following special cases.
\begin{Example}
For the  commonly studied intervention operator (see e.g.~\cite{K.ishii1993,bonnans2007,davis2010,sotomayor2013,azimzadeh2016weakly,azimzadeh2017zero}):
\bb\l{eq:M_common}
\cM_iu(x)=\inf_{z\in \R^p}[u_i(x+\gamma(z))+K(z)], \q x\in \R^d,
\ee
with $|K(z)|\to \infty$ as $|z|\to \infty$, it is straightforward to show that   (H.\ref{assum:M_lip}) and (H.\ref{assum:M_concave}) hold with $C=0$ (note the growth of $K$ ensures the optimal impulse strategy is attained in a compact set). See Section \ref{sec:switching} for examples with state-dependent impulse costs. 
\end{Example}

%

\begin{Example}
For the intervention operator $\cM u(x)=\inf_{z\in Z(x)} [u(x-z)+K(z)]$, with $x\in \R_+^d\coloneqq (0,\infty)^d$ and $Z(x)=\{z\in \R^d\mid 0\le z_i\le x_i,\, i=1,\ldots,d\}$, which is a concave analogue of the maximum utility operator for  multi-dimensional optimal dividend/inventory problems (e.g.~\cite{bai2010}), one can show that   (H.\ref{assum:M_lip}) (resp.~(H.\ref{assum:M_concave})) holds if $K$ is Lipschitz continuous (resp.~semiconcave).

We shall only discuss (H.\ref{assum:M_concave}), since (H.\ref{assum:M_lip}) can be shown by a similar approach. Let $x\in \R_+^d$ and $\hat{z}\in Z(x)$ such that $\cM u(x)= u(x-\hat{z})+K(\hat{z})$. Define the set $\I_x=\{1\le i\le d\mid \hat{z}_i=x_i\}$ and the constant $h_0=\min(\min_{i\not \in \I_x} (x_i-\hat{z}_i), \min_{ i=1,\ldots, d} x_i)>0$. Then for any given  $h\in \R^d$ such that $|h|_0< h_0$, we can consider the  vector $h^{ \I_x}=(h^{ \I_x}_i)_{i=1}^d$ defined by $h^{ \I_x}_i=h_i$ if $i\in \I_x$ and $0$ otherwise,  which  satisfies the following properties:
$$
0\le \hat{z}_i+h^{ \I_x}_i\le x_i+h_i,\q 0\le \hat{z}_i-h^{ \I_x}_i\le x_i-h_i,\q \fa i\in \cI.
$$
In other words, we have $z^+\coloneqq \hat{z}+h^{ \I_x}\in Z(x+h)$ and $z^-\coloneqq \hat{z}-h^{ \I_x}\in Z(x-h)$.
Therefore one can deduce from the semiconcavity of $u$ and $K$ that  $\cM u$ is semiconcave around $x$:
\begin{align*}
\cM u(x+h)-&2\cM u(x)+\cM u(x-h)\\
\le\,& u(x+h-z^+)+K(z^+)-2(u(x-\hat{z})+K(\hat{z}))+u(x-h-z^-)+K(z^-)\\
\le\,& [u]_{2,+}|h-h^{ \I_x}|^2+[K]_{2,+}|h^{ \I_x}|^2\le ([u]_{2,+}+[K]_{2,+})|h|^2,
\end{align*}
which subsequently leads to the desired estimate $[\cM u]_{2,+}\le [u]_{2,+}+[K]_{2,+}$.
\end{Example}

\begin{Example}
A general intervention operator  \eqref{eq:M_impulse} satisfies (H.\ref{assum:M_lip}) under the following Lipschitz conditions on the data:  there exist constants $C_1,C_2,C_3,C_4\ge 0$ such that $C_2+C_1C_3\le 1$ and 
\begin{align*}
&Z_i(y)\subseteq Z_i(x)+\bar{B}_{C_1|x-y|},\q |\Gamma_i(x,z)-\Gamma_i(y,z')|\le C_2|x-y|+C_3d_{\boldsymbol{Z}}(z,z'),\\
  &|K_i(x,z)-K_i(y,z')|\le C_4(|x-y|+d_{\boldsymbol{Z}}(z,z')), \q\q \fa i\in \cI, x,y\in \R^d, z,z'\in \boldsymbol{Z},
\end{align*}
where for each $r\ge 0$, $\bar{B}_r$ denotes a closed ball of center $0$ and radius $r$ in the metric space $\boldsymbol{Z}$.  In fact, let $i\in \cI$, $u\in C_1^0(\R^d)$, $x,y\in \R^d$, $\hat{z}(y)\in Z(y)$ such that $(\mathcal{M}_i u)(y) =  u(\Gamma_i(y,\hat{z}(y))) + K_i(y,\hat{z}(y))$. Then we can find $z(x)\in Z_i(x)$ such that $d_{\boldsymbol{Z}}({z}(x),\hat{z}(y)) \le C_1|x-y|$, which leads to the following estimate that 
\begin{align*}
&(\mathcal{M}_i u)(x)-(\mathcal{M}_i u)(y)\\
&\le  \big[ u(\Gamma_i(x,{z}(x))) + K_i(x,{z}(x))\big]-\big[u(\Gamma_i(y,\hat{z}(y))) + K_i(y,\hat{z}(y))\big]\\
&\le [u]_1|\Gamma_i(x,{z}(x))-\Gamma_i(y,\hat{z}(y))|+|K_i(x,{z}(x))-K_i(y,\hat{z}(y))|\\
&\le [u]_1\big(C_2|x-y|+C_3d_{\boldsymbol{Z}}(z(x),\hat{z}(y))\big)+C_4(|x-y|+d_{\boldsymbol{Z}}(z(x),\hat{z}(y)))\\
&\le \big([u]_1\big(C_2+C_1C_3)+C_4(1+C_1)\big)|x-y|.
\end{align*}
Hence we can conclude from the assumption $C_2+C_1C_3\le 1$ that $\mathcal{M}_i$ satisfies (H.\ref{assum:M_lip}) 
with $C=C_4(1+C_1)$. A sufficient condition of (H.\ref{assum:M_concave}) for the intervention operator  \eqref{eq:M_impulse} in general involves technical second-order conditions on the set-valued mapping $x\to Z(x)$, which  will not be derived here for the sake of simplicity. 
\end{Example}
\color{black}

%
%
%
Note that we do not require any non-degeneracy condition on the diffusion coefficients, i.e., the coefficient  $a^\a_i$ may vanish at certain points, hence our results apply to the fully degenerate case with $a^\a = 0$, where \eqref{eq:qvi_impulse} reduces to QVIs of first order. 

We now discuss the well-posedness of HJBVI \eqref{eq:qvi_impulse}. Due to the lack of  regularization from a Laplacian operator, the solution of \eqref{eq:qvi_impulse} is typically nonsmooth and we shall understand all equations in this work in the following viscosity sense. 

\begin{Definition}[Viscosity solution] \label{def:sol}
A bounded, upper-semicontinuous (resp. lower-semicontinuous) function $u=(u_i)_{i\in \cI}$
is a viscosity subsolution (resp.~supersolution) to \eqref{eq:qvi_impulse}, if for each $i\in \cI$ and  function $\phi \in C^{2}(\R^d)$, at each local maximum (resp.~minimum) point $x$ of $u_i - \phi$ we have $F_i(x, u(x),  D \phi(x), D^2\phi(x)) \leq 0$ (resp.~$\ge 0$).
A continuous function  is a viscosity solution of \eqref{eq:qvi_impulse} if it is both a subsolution and a supersolution.
\end{Definition}

\begin{Remark}
Definition \ref{def:sol} formulates the notation of viscosity solution with suitable  test functions. It is well-known that one can equivalently define the viscosity solution to \eqref{eq:qvi_impulse} in terms of the superjet and subjet of $u_i$ at $x\in \R^d$, denoted by  $J^{2,+}u_i(x)$ and $J^{2,-}u_i(x)$ respectively, or their closures $\bar{J}^{2,+}u_i(x)$ and $\bar{J}^{2,-}u_i(x)$ (see e.g.~\cite[Proposition~2.3]{ishii1991monotone}).
\end{Remark}

The fact that the impulse cost is strictly positive (see  (H.\ref{assum:impulse})) implies  $-C$ is a strict subsolution to \eqref{eq:qvi_impulse} for a  large enough constant $C>0$.
Therefore, one can  establish  a  comparison principle of \eqref{eq:qvi_impulse} by using similar arguments as in \cite{K.ishii1993} (cf.~the proof of Proposition \ref{prop:penalty_comparison}; see also \cite[Theorem 2.5.11]{seydel2009} for a related result for solutions of polynomial growth). The comparison principle directly leads to the uniqueness of bounded viscosity solutions to \eqref{eq:qvi_impulse}, which can be explicitly constructed through penalty approximations (Theorem \ref{thm:conv_penalty}).

\begin{Proposition}\l{prop:impulse_comparison}
Suppose (H.\ref{assum:regularity}) and (H.\ref{assum:impulse}) hold. If $u$ (resp.~$v$) is a bounded subsolution (resp.~supersolution) of \eqref{eq:qvi_impulse}, then $u\le v$ in $\R^d$.
\end{Proposition}

We end this section by collecting several important properties of the intervention operator $\cM_i$.

\begin{Lemma}\l{Lemma:cM}
For any $i\in \cI$, we have: 
\begin{enumerate}[(1)]
\item $\cM_i$ is concave, i.e., it holds for any locally bounded functions $u, v:\R^d\to \R$ and constant $\lambda\in [0,1]$ that $\cM_i[(1-\lambda)u+\lambda v]\ge (1-\lambda)\cM_iu+\lambda \cM_iv$.

\item $\cM_i$ is monotone, i.e., if $u\ge v$, then $\cM_i u\ge \cM_i v$.
\item Suppose (H.\ref{assum:impulse}) holds, and let $(u^\rho)_{\rho\in \N}$ be a family of uniformly bounded functions on $\R^d$ with the following half-relaxed limits $u^*$ and $u_*$: 
\bb\l{eq:relaxed}
u^*(x)\coloneqq \limsup_{\rho\to \infty, y\to x}u^\rho(y), \q u_*(x)\coloneqq \liminf_{\rho\to \infty, y\to x}u^\rho(y), \q x\in \R^d.
\ee
Then it holds for any given $x\in \R^d$ and sequence $(x^\rho)_{\rho\in \N}$ with $\lim_{\rho\to \infty} x^\rho=x$ that 
\bb
(\cM_i u_* )(x)\le \liminf_{\rho \to \infty} (\cM_i u^\rho)(x^\rho)\le \limsup_{\rho \to \infty} (\cM_i u^\rho)(x^\rho)\le (\cM_i u^*) (x).
\ee
\end{enumerate} 
\end{Lemma}
\begin{proof}
Properties (1) and (2) follow directly from the structure of $\cM_i$. Property (3) 
is an analogue   of   \cite[Lemma 12]{azimazadeh2018} to the present concave intervention operator $\cM_i$ and compact set $Z_i$, 
whose proof will be given in Appendix \ref{appendix} for completeness.
\end{proof}

\section{Penalty approximations for HJBQVIs}\l{sec:penalty}
In this section, we  propose a penalty approximation for the system of HJBQVIs \eqref{eq:qvi_impulse},
which is an extension of the ideas used for scalar HJBQVIs in \cite{bensoussan1982,azimazadeh2018}.
We shall also establish the monotone convergence of the penalized solutions in terms of the penalty parameter. 

For any given penalty parameter $\rho\ge 0$, we consider the following  system of HJB equations: for all $i\in \cI$ and $x\in \R^d$, 
\begin{align}\l{eq:penalty}
\begin{split}
&F^\rho_i(x,u^\rho,Du^\rho_i,D^2u^\rho_i)\\
&\coloneqq 
\sup_{\alpha\in\cA_i}\mathcal{L}^{\alpha}_i(x,u^\rho(x),Du^\rho_i(x),D^2u^\rho_i(x))+\rho(u^\rho_i-\cM_i u^\rho)^+(x)=0, 
\end{split}
\end{align}
where the operators $\cL^\a_i$ and $\cM_i$ are defined as in \eqref{eq:defLi} and \eqref{eq:M_impulse}, respectively.

The definitions of viscosity solution, sub- and supersolution for \eqref{eq:penalty} extend naturally from Definition \ref{def:sol}. 
The following result asserts the comparison principle and the well-posedness of \eqref{eq:penalty} for any given penalty parameter.
\begin{Proposition}\l{prop:penalty_comparison}
Suppose (H.\ref{assum:regularity}) and (H.\ref{assum:impulse}) hold, and let $\rho\ge 0$ be a given penalty parameter. If $u^\rho$ (resp.~$v^\rho$) is a bounded subsolution (resp.~supersolution) of \eqref{eq:penalty}, then $u^\rho\le v^\rho$ in $\R^d$. Consequently,  \eqref{eq:penalty} admits a unique   viscosity solution, which is uniformly bounded in $\rho$.
\end{Proposition}
\begin{proof}
We postpone the proof of the comparison principle to Appendix \ref{appendix}, which adapts the strict subsolution technique in \cite{K.ishii1993} to the penalized equation, and  reduces the problem to a HJB equation without the penalty part.

Since $K(x,z)\ge \kappa_0>0$, there exits a large enough constant $C$,  independent of $\rho$, such that
 $-C$ and $C$ are  the viscosity sub- and supersolution of \eqref{eq:penalty} with any parameter $\rho$, respectively. Thus by using the comparison principle and  Perron's method (see \cite[Theorem 3.3]{ishii1991monotone}), one can deduce the well-posedness of \eqref{eq:penalty} in the viscosity sense. 
\end{proof}

The next result demonstrates the monotone and locally uniform convergence of the solution $(u^\rho)_{\rho\ge 0}$ of \eqref{eq:penalty} in terms of the penalty parameter $\rho$.
\begin{Theorem}\l{thm:conv_penalty}
Suppose  (H.\ref{assum:regularity}) and (H.\ref{assum:impulse}) hold. 
Then as $\rho\to \infty$, the  solution   of \eqref{eq:penalty}  converges monotonically from above to the bounded viscosity solution of \eqref{eq:qvi_impulse}, uniformly on compact sets.
\end{Theorem}
\begin{proof}
It is clear that if $\rho_1 \le \rho_2$ and $u^{\rho_2}$ is a subsolution to \eqref{eq:penalty} with the parameter $\rho_2$, then $u^{\rho_2}$ is a subsolution to \eqref{eq:penalty} with the parameter $\rho_1$. Hence the comparison principle leads to the fact that $u^0\ge u^{\rho_1}\ge u^{\rho_2}$. 
Now we shall adopt the equivalent definition of viscosity solution in terms of  semi-jets and prove that  the component-wise half-relaxed limit  $u^*$ (resp.~$u_*$) is a subsolution (resp.~supersolution) to \eqref{eq:qvi_impulse}.

We start by showing $u^*$ is a subsolution. Let $x\in \R^d$, $i\in \cI$ and $(p,X)\in J^{2,+}{u}^*_i(x)$, then  by applying \cite[Lemma~6.1]{crandall1992}, there exist sequences  $(x^\rho,p^\rho,X^\rho)_{\rho\in \N}$ such that $(p^\rho,X^\rho)\in J^{2,+}{u}^\rho_i(x^\rho)$ for each $\rho$ and $(x^\rho,{u}^\rho_i(x^\rho),p^\rho,X^\rho)\to (x,u^*_i(x),p,X)$ as $\rho\to\infty$. Since $u^\rho$ is a subsolution to \eqref{eq:penalty}, we have
\bb\l{eq:impluse_sub}
\sup_{\alpha\in\cA_i}\mathcal{L}^{\alpha}_i(x^\rho,u^\rho(x^\rho),p^\rho,X^\rho)+\rho(u^\rho_i-\cM_iu^\rho)^+(x^\rho)\le 0, \q \fa \rho\in \N.
\ee
Then it follows from the boundedness of coefficients that there exists a constant $C>0$ such that
$$
 u^\rho_i(x^\rho)-\cM_i u^\rho(x^\rho)\le (u^\rho_i-\cM_iu^\rho)^+(x^\rho)\le C/\rho,
$$
hence by letting $\rho\to \infty$ and using Lemma \ref{Lemma:cM} (3), we deduce that 
$$
u^*_i(x)= \lim_{\rho\to \infty} u^\rho_i(x^\rho)\le \limsup_{\rho\to \infty} (\cM_i u^\rho(x^\rho)+C/\rho)\le 
\cM_i u^*(x).
$$
On the other hand, \eqref{eq:impluse_sub} yields for any $\a\in \cA_i$, $\mathcal{L}^{\alpha}_i(x^\rho,u^\rho(x^\rho),p^\rho,X^\rho)\le 0$, which implies  that
\bb\l{eq:liminf_conv}
\begin{split}
&-\mathrm{tr}[a^{\alpha}_i(x) X] -
b^{\alpha}_i(x) p +c^\a_i(x)u^*_i(x)-\ell_i^\a (x)\\
&\le \sum_{j\in \cI^{-i}}\limsup_{\rho\to \infty}d_{ij}^\a(x^\rho)u^\rho_j(x^\rho)\le  \sum_{j\in \cI^{-i}}d_{ij}^\a(x)u^*_j(x),
\end{split}
\ee
where we have used the fact that $\lim_{\rho\to \infty}d_{ij}^\a(x^\rho)=d_{ij}^\a(x)\ge 0$. Then by taking the supremum over $\a$, we have
$\sup_{\alpha\in\cA_i}\mathcal{L}^{\alpha}_i(x,u^*(x),p,X)\le 0$,
which shows  $u^*$ is a subsolution to \eqref{eq:qvi_impulse}. 

Then we proceed to study $u_*$ by fixing $x\in \R^d$, $i\in \cI$ and $(p,X)\in J^{2,-}(u_*)_i(x)$. Let $(x^\rho,p^\rho,X^\rho)_{\rho\in \N}$ be a sequence  such that $(p^\rho,X^\rho)\in J^{2,-}{u}^\rho_i(x^\rho)$ for each $\rho$ and $(x^\rho,{u}^\rho_i(x^\rho),p^\rho,X^\rho)\to (x,(u_*)_i(x),p,X)$ as $\rho\to\infty$. Then the supersolution property of $u^\rho$ implies for each $n\in \N$,
\bb\l{eq:impluse_sup}
\sup_{\alpha\in\cA_i}\mathcal{L}^{\alpha}_i(x^\rho,u^\rho(x^\rho),p^\rho,X^\rho)+\rho(u^\rho_i-\cM_iu^\rho)^+(x^\rho)\ge 0.
\ee

Suppose that $\limsup_{\rho\to \infty}\rho(u^\rho_i-\cM_iu^\rho)^+(x^\rho)>0$, then, by possibly passing to a subsequence, we have $u^\rho_i(x^\rho)>(\cM_iu^\rho)(x^\rho)$ for all $\rho$. Then, by using Lemma \ref{Lemma:cM} (3), we obtain that
$(u_*)_i(x)\ge \liminf_{\rho\to \infty}(\cM_iu^\rho)(x^\rho)\ge (\cM_iu_*)(x)$.

On the other hand, suppose that $\limsup_{\rho\to \infty}\rho(u^\rho_i-\cM_iu^\rho)^+(x^\rho)=0$, then for any $\delta>0$ , by  passing to a subsequence, we deduce that for large enough $\rho\in \N$, there exists $\a^{\rho,\delta}\in \cA_i$ such that
$$
-\mathrm{tr}[a^{\alpha^{\rho,\delta}}_i(x^\rho) X^\rho] -
b^{\alpha^{\rho,\delta}}_i(x^\rho) p^\rho +c^{\a^{\rho,\delta}}_i(x^\rho)u^\rho_i(x^\rho)-\ell_i^{\a^{\rho,\delta}} (x^\rho)-  \sum_{j\in \cI^{-i}}d_{ij}^{\a^{\rho,\delta}}(x^\rho)u^\rho_j(x^\rho)\ge\! -\delta.
$$
Since $\cA_i$ is compact, we can assume $\a^{\rho,\delta}\to \a^{\delta}\in \cA_i$ as $\rho\to \infty$. Then, by taking the  limit inferior and using the fact that $\liminf_{\rho\to \infty}d_{ij}^{\a^{\rho,\delta}}(x^\rho)u^\rho_j(x^\rho)\ge d_{ij}^{\a^\delta}(x)(u_*)_j(x)$, we deduce that 
\bb\l{eq:liminf_conv2}
-\mathrm{tr}[a^{\alpha^\delta}_i(x) X] -
b^{\alpha^\delta}_i(x) p +c^{\a^\delta}_i(x)(u_*)_i(x)-\ell_i^{\a^\delta} (x)-\sum_{j\in \cI^{-i}}d_{ij}^{\a^\delta}(x)(u_*)_j(x)\ge -\delta,
\ee
from which, by taking the supremum over $\a$ and sending $\delta\to 0$, we can deduce that 
$\sup_{\alpha\in\cA_i}\mathcal{L}^{\alpha}_i(x,u^*(x),p,X)\ge 0$, and conclude that  $u_*$ is a supersolution to \eqref{eq:qvi_impulse}.

Finally, by using Proposition \ref{prop:impulse_comparison}, we have 
$u\coloneqq {u}^*={u}_*$ is the unique continuous viscosity solution of \eqref{eq:qvi_impulse} and consequently $(u^\rho)$ converges to $u$ locally uniformly. 
\end{proof}
\begin{Remark}
Theorem \ref{thm:conv_penalty} provides us with a constructive proof for the existence of solutions of \eqref{eq:qvi_impulse} based on penalty approximations. Moreover, since the convergence analysis relies only on the comparison principle of \eqref{eq:qvi_impulse} and the local boundedness of $(u^\rho)_{\rho\ge 0}$, it is possible to extend the results to nonlocal non-convex/non-concave systems  with   coefficients of polynomial growth.

\end{Remark}

\section{Error estimates for penalty approximations}\l{sec:error estimate}
In this section, we shall proceed to  analyze the convergence rate of the penalty approximation for \eqref{eq:qvi_impulse}.
As pointed out in Section \ref{sec:intro}, unlike the  variational inequalities \cite{jakobsen2006},  the non-strict monotonicity of the term $u_i-\cM_iu$ prevents us from obtaining an upper bound of $u^\rho-u$ by  constructing a subsolution of \eqref{eq:qvi_impulse}  directly from the penalized equations, which significantly complicates the error analysis. We shall overcome this difficulty by regularizing the HJBQVIs, and recover the same convergence rates (up to a logarithmic term) as those for conventional obstacle problems. 

\subsection{Regularization of HJBQVIs}\l{sec:regularize}
In this section, we  approximate \eqref{eq:qvi_impulse} by  a sequence of obstacle problems, through the iterated optimal stopping approximation (see e.g.~\cite{lions1982,seydel2009,ferretti2017} for its application  to QVIs). We shall   quantify the approximation errors of these obstacle problems depending on the regularity of the solution, which we also establish.

Let  $u^0=(u_i^0)_{i\in\cI}$ be the viscosity solution of the following system of HJB equations:
\begin{align}
\label{eq:iter_0}
\sup_{\alpha\in\cA_i}\mathcal{L}^{\alpha}_i(x,u(x),Du_i(x),D^2u_i(x))=0,\q x\in \R^d, \, i\in\cI.
\end{align}
We then  inductively define a sequence of functions $\{u^n\}_{n\in \N}$, where for each $n\in \N$, i.e., $n>0$, given functions $u^{n-1}$, let $u^n=(u^{n}_i)_{i\in\cI}$ be the viscosity solution to the following obstacle problem:
\bb\l{eq:iter_n}
\max\Big\{
\sup_{\alpha\in\cA_i}\mathcal{L}^{\alpha}_i(x,u^n(x),Du^n_i(x),D^2u^n_i(x)), \   
(u^n_i-\mathcal{M}_i u^{n-1})(x)\Big\}=0,  \q x\in \R^d, \, i\in\cI.
\ee

Under the assumptions (H.\ref{assum:regularity})--(H.\ref{assum:lambda_lip}) and (H.\ref{assum:M_lip}), one can establish the comparison principles for \eqref{eq:iter_0} and
\eqref{eq:iter_n}, and then demonstrate the existence of $u^n\in [C^0_1(\R^d)]^M$ for each $n\ge 0$ (see Theorem \ref{thm:iter_regularity} for the Lipschitz regularity). Moreover, by using the comparison principle of \eqref{eq:iter_n}, we can further deduce from an inductive argument that $u^{n-1}\ge u^n$ for all $n\in \N$.

%
%

The following proposition estimates the  approximation error $u^n-u$, which extends the results in \cite{bonnans2007,ferretti2017} to weakly coupled systems with (possibly) negative running cost $(\ell_i)_{i\in\cI}$.

\begin{Proposition}\l{prop:iter_conv}
Suppose (H.\ref{assum:regularity})--(H.\ref{assum:lambda_lip}) and (H.\ref{assum:M_lip}) hold, and let $u$ and $u^n$ be the viscosity solution to \eqref{eq:qvi_impulse} and \eqref{eq:iter_n}, respectively. 
Then there exist  constants $\mu\in (0,1]$ and $C\ge 0$ such that 
$$
0\le u^n-u\le C{(1-\mu)^n},\q n\ge 0.
$$
Consequently, the iterates $(u^n)_{n\ge 0}$ are bounded uniformly in $n$.
\end{Proposition}
\begin{proof}
We adapt the arguments for  \cite[Theorem 3.4]{reisinger2018qvi} to the current continuous setting, and  present the main steps in Appendix \ref{appendix} for the reader's convenience. 
\end{proof}

Now we turn to investigate the regularity of solutions to \eqref{eq:iter_n} based on different assumptions on the coefficients. We shall first focus on the following variational inequalities:
\bb\l{eq:iter_ob}
\max\Big\{
\sup_{\alpha\in\cA_i}\mathcal{L}^{\alpha}_i(x,u(x),Du_i(x),D^2u_i(x)), \   
(u_i-\Psi_i)(x)\Big\}=0, \q i\in \cI,
\ee
with  given obstacles $(\Psi_i)_{i\in \cI}$, which serves as a general form of the iterative equations \eqref{eq:iter_n}. 

The following result shows the Lipschitz continuity of the solution to the obstacle problem \eqref{eq:iter_ob}, which can be proved by using the standard doubling of variables technique (see e.g.~\cite{briani2012,ferretti2017}).
\begin{Proposition}\l{prop:ob_Lip}
Suppose (H.\ref{assum:regularity}) and (H.\ref{assum:lambda_lip}) hold, and $\Psi\in [C^0_1(\R^d)]^M$. 
Then the viscosity solution $u$ to  \eqref{eq:iter_ob} is Lipschitz continuous with constant $\sup_{i\in \cI}[u_i]_1\le \max(C,\sup_{i\in \cI}[\Psi_i]_1)$, where $C$ is a constant independent of $[\Psi_i]_1$.
\end{Proposition}


We then proceed to study higher  regularity of the solutions, which enables us to deduce a higher convergence rate of the penalty approximation. 
The next proposition extends the results in \cite{ishii1990,ishii1995} to weakly coupled systems, and asserts that  if the coefficients are sufficiently regular, then the  solution to the  obstacle problem \eqref{eq:iter_ob} is  semiconcave.

Note that instead of viewing the obstacle problem \eqref{eq:iter_ob} as a convex  HJB equation as is studied in \cite{ishii1990,ishii1995},  we shall separately analyze the obstacle part and the HJB part of \eqref{eq:iter_ob}, which leads to a sharper estimate for  the semiconcavity constant of $u$ in terms of $\Psi$. Moreover, instead of requiring $\Psi_i\in W^{2,\infty}(\R^d)$ as in \cite{ishii1990,ishii1995} (which essentially means $\Psi_i$ is  differentiable with bounded and Lipschitz continuous derivative),  we only assume the obstacles   to be semiconcave, which is crucial for  the subsequent analysis of penalty errors. 

\begin{Proposition}\l{prop:semiconcave_ob}
Suppose (H.\ref{assum:regularity}) and  (H.\ref{assum:regularity_concave})  hold. Assume further that the constant $\lambda_0$ in (H.\ref{assum:regularity}) is sufficiently large and the obstacle $\Psi\in [C^0_1(\R^d)]^M$ is semiconcave.
Then the viscosity solution  $u\in [C^0_1(\R^d)]^M$ to \eqref{eq:iter_ob}  is semiconcave with a  constant  satisfying
the  estimate
$$
\sup_{i\in \cI}[u_i]_{2,+}\le \max\bigg\{C\sup_{i\in \cI} |u_i|_1, \sup_{i\in \cI} [\Psi_i]_{2,+}\bigg\}, 
$$
for some constant $C$, independent of $[\Psi_j]_{2,+}$, $[\Psi_j]_{1}$ and $[u_j]_1$ for all $j\in \cI$.
\end{Proposition}
\begin{proof}
For $\delta,\eps,\gamma>0$, we define for all $x,y,z\in \R^d$ that
\begin{align*}
\phi(x,y,z)&=\delta|x-y|^4+ \eps |x+y-2z|^2+\gamma |x|^2,\\
\Phi_i(x,y,z)&=u_i(x)+u_i(y)-2u_i(z)-\phi(x,y,z),
\end{align*}
and let $m_{\delta,\eps,\gamma}\coloneqq \sup_{(x,y,z)\in \R^{3d}, i\in \cI}\Phi_i(x,y,z)$. 
By the finiteness of $\cI$,  the boundedness and continuity of $(u_i)_{i\in\cI}$, and the penalization term $\phi$, 
 there exists $i\in \cI$, independent of $\delta,\eps,\gamma$, and $(\bar{x}^{\delta,\eps,\gamma},\bar{y}^{\delta,\eps,\gamma},\bar{z}^{\delta,\eps,\gamma})\in \R^{3d}$ such that $m_{\delta,\eps,\gamma}=\Phi_i(\bar{x}^{\delta,\eps,\gamma},\bar{y}^{\delta,\eps,\gamma},\bar{z}^{\delta,\eps,\gamma})$. In the following we shall  omit the dependence on $\delta,\eps,\gamma$ for notational simplicity. Then   we can deduce from  \cite[Theorem 3.2]{crandall1992} that for any $\theta>1$, there exist  $X, Y, Z\in \bS^d$ such that
$$
(p_x, X)\in \bar{J}^{2,+}u(\bar{x}),\q (p_y,Y)\in \bar{J}^{2,+}u(\bar{y}), \q (-p_z/2,-Z/2)\in \bar{J}^{2,-}u(\bar{z}),
$$
where $(p_x,p_y,p_z)=(D_x\phi(\bar{x},\bar{y},\bar{z}), D_y\phi(\bar{x},\bar{y},\bar{z}),D_z\phi(\bar{x},\bar{y},\bar{z}))$, and 
$$
\begin{pmatrix} X &0 &0\\ 0& Y &0\\ 0& 0&Z \end{pmatrix}\le \theta D^2\phi (\bar{x},\bar{y},\bar{z}).
$$
Hence, by the definition of viscosity solution, we obtain that
\begin{align}\l{eq:three_sub}
\begin{split}
\max\Big\{
\sup_{\alpha\in\cA_i}\mathcal{L}^{\alpha}_i(\bar{x},u(\bar{x}),p_x,X), \   
(u_i-\Psi_i)(\bar{x})\Big\}&\le 0,\\
\max\Big\{\sup_{\alpha\in\cA_i}\mathcal{L}^{\alpha}_i(\bar{y},u(\bar{y}), p_y,Y), \   
(u_i-\Psi_i)(\bar{y})\Big\}&\le 0,\\
\max\Big\{\sup_{\alpha\in\cA_i}\mathcal{L}^{\alpha}_i(\bar{z},u(\bar{z}),-p_z/2,-Z/2), \   
(u_i-\Psi_i)(\bar{z})\Big\}&\ge 0. 
\end{split}
\end{align}

Now we discuss two cases. Suppose the maximum in the third inequality of \eqref{eq:three_sub} is attained by its second argument, then 
we obtain from \eqref{eq:three_sub} that
\begin{align}
&u_i(\bar{x})+u_i(\bar{y})-2u_i(\bar{z})\nb\\
&\le \Psi_i(\bar{x})+\Psi_i(\bar{y})-2\Psi_i(\bar{z})
= \Psi_i(\bar{x})+\Psi_i(\bar{y})-2\Psi_i(\f{\bar{x}+\bar{y}}{2})+2\Psi_i(\f{\bar{x}+\bar{y}}{2})-2\Psi_i(\bar{z}) \nb\\
&\le [\Psi_i]_{2,+}|\bar{x}-\bar{y}|^2/4+[\Psi_i]_1|\bar{x}+\bar{y}-2\bar{z}|, \nb
\end{align}
where we have used the Lipschitz continuity and semiconcavity of $\Psi_i$. Thus, the definition of $m_{\delta,\eps,\gamma}$ and the fact that $\sup_{r> 0}(-\delta r^2+Cr)=C^2/(4\delta)$ give us that
\begin{align*}
m_{\delta,\eps,\gamma}&\le [\Psi_i]_{2,+}|\bar{x}-\bar{y}|^2/4+[\Psi_i]_1|\bar{x}+\bar{y}-2\bar{z}|-\delta|\bar{x}-\bar{y}|^4-\eps |\bar{x}+\bar{y}-2\bar{z}|^2\\
&\le \f{[\Psi_i]_{2,+}^2}{64\delta}+\f{[\Psi_i]_{1}^2}{4\eps}.
\end{align*}
Thus, by letting $\gamma\to \infty$, 
we have for all  $x,y,z\in \R^d$ that
\begin{align*}
u_i(x)+u_i(y)-2u_i(z)\le \f{[\Psi_i]_{2,+}^2}{64\delta}+\f{[\Psi_i]_{1}^2}{4\eps}+\delta|x-y|^4+ \eps|x+y-2z|^2, \q \fa \delta, \eps>0,
\end{align*}
from which by  minimizing over $\delta,\eps$ separately, and setting $x=z+h$, $y=z-h$,  we obtain that
\bb\l{eq:semi_1}
u_i(z+h)+u_i(z-h)-2u_i(z)\le  {[\Psi_i]_{2,+}}|2h|^2/4= [\Psi_i]_{2,+}|h|^2.
\ee

On the other hand, suppose the maximum in  the third inequality of \eqref{eq:three_sub}   is attained by the first argument, then 
for any $\eta>0$, there exists $\a^\eta\in \cA$ such that the following inequality holds:
$$
\mathcal{L}^{\alpha^\eta}_i(\bar{x},u(\bar{x}),p_x,X)+\mathcal{L}^{\alpha^\eta}_i(\bar{y},u(\bar{y}), p_y,Y)-2(\mathcal{L}^{\alpha^\eta}_i(\bar{z},u(\bar{z}),-p_z/2,-Z/2)+\eta)\le 0.
$$
More precisely, we have
\begin{align*}
\begin{split}
&-\mathrm{tr}[a^{\alpha^\eta}_i(\bar{x}) X+a^{\alpha^\eta}_i(\bar{y}) Y+a^{\alpha^\eta}_i(\bar{z}) Z] -
[b^{\alpha^\eta}_i(\bar{x}) p_x+b^{\alpha^\eta}_i(\bar{y}) p_y+b^{\alpha^\eta}_i(\bar{z}) p_z]\\
&+c^{\a^\eta}_i(\bar{x})u_i(\bar{x})+c^{\a^\eta}_i(\bar{y})u_i(\bar{y})-2c^{\a^\eta}_i(\bar{z})u_i(\bar{z}) -[\ell^{\a^\eta}_i(\bar{x})+\ell^{\a^\eta}_i(\bar{y})-2\ell^{\a^\eta}_i(\bar{z})]-2\eta\\
&\le \sum_{j\in \cI^{-i}}d^{\a^\eta}_{ij}(\bar{x})u_j(\bar{x})+d^{\a^\eta}_{ij}(\bar{y})u_j(\bar{y})-2d^{\a^\eta}_{ij}(\bar{z})u_j(\bar{z}).
\end{split}
\end{align*}
Comparing with \cite[Theorem 5 (ii)]{ishii1995}, it remains  to  estimate the terms in the last line  of the above inequality.
Note for any given function $g\in C_1^0(\R^d)$ and $x,z\in \R^d$, we have that
\begin{align*}
&|g(x)-g(z)|\\
&\le \big|g(x)-g(\f{x+y}{2})\big|+\big|g(\f{x+y}{2})-g(z)\big|
\le [g]_1\f{|x-y|}{2}+(|g|_0[g]_1|x+y-2z|)^{1/2}\\
&\le |g|_1 (|x-y|+|x+y-2z|^{1/2}),\q \fa y\in \R^d.
\end{align*}
Therefore, 
we obtain for each $j\in \cI^{-i}$ that
\begin{align*}
&d^{\a^\eta}_{ij}(\bar{x})u_j(\bar{x})\,+\, d^{\a^\eta}_{ij}(\bar{y})u_j(\bar{y})-2d^{\a^\eta}_{ij}(\bar{z})u_j(\bar{z})\\
=\,&d^{\a^\eta}_{ij}\big(\bar{z})(u_j(\bar{x})+u_j(\bar{y})-2u_j(\bar{z})\big)+\big(d^{\a^\eta}_{ij}(\bar{x})+d^{\a^\eta}_{ij}(\bar{y})-2d^{\a^\eta}_{ij}(\bar{z})\big)u_j(\bar{z})\\
&+(d^{\a^\eta}_{ij}(\bar{x})-d^{\a^\eta}_{ij}(\bar{z}))(u_j(\bar{x})-u_j(\bar{z}))+(d^{\a^\eta}_{ij}(\bar{y})-d^{\a^\eta}_{ij}(\bar{z}))(u_j(\bar{y})-u_j(\bar{z}))\\
\le \,&d^{\a^\eta}_{ij}\big(\bar{z})(u_i(\bar{x})+u_i(\bar{y}\big)-2u_i(\bar{z}))+\big([Dd^{\a^\eta}_{ij}]_{1}|\bar{x}-\bar{y}|^2/4+[d^{\a^\eta}_{ij}]_1|\bar{x}+\bar{y}-2\bar{z}|\big)|u_j|_0\\
&+4|d^{\a^\eta}_{ij}|_1|u_j|_1 (|x-y|^2+|x+y-2z|).
\end{align*}
Then if $\lambda_0$ is sufficiently large, we can proceed  along lines of the proof of \cite[Theorem 5 (ii)]{ishii1995}, and deduce that there exists a constant $C\ge 0$, independent of  $[u_i]_1$ for any $i\in \cI$, such that it holds  for all $z,h\in \R^d$ and $i\in \cI$ that
$u_i(z+h)+u_i(z-h)-2u_i(z)\le C(\sup_{i\in \cI} |u_i|_1)|h|^2$ (cf.~equation (5.8) in \cite{ishii1995}),
which together with \eqref{eq:semi_1} completes our proof.
\end{proof}

With Propositions \ref{prop:iter_conv}  and \ref{prop:semiconcave_ob} in hand, we are ready to present the following upper bounds of the Lipschitz and semiconcavity constants of the iterates $(u^n)_{n\in \N}$ defined as in \eqref{eq:iter_0} and \eqref{eq:iter_n}.
\begin{Theorem}\l{thm:iter_regularity}
Suppose (H.\ref{assum:regularity})--(H.\ref{assum:lambda_lip}) and (H.\ref{assum:M_lip}) hold, then for any $n\in \N$, the iterate $u^n$   is Lipschitz continuous with a constant  satisfying 
$\sup_{i\in \cI}[u^n_i]_{1} \le Cn$, where  $C$ is a constant independent of $n$. If we further assume (H.\ref{assum:regularity_concave}) and (H.\ref{assum:M_concave}) hold, and the constant $\lambda_0$ in (H.\ref{assum:regularity}) is sufficiently large,  then
the iterate $u^n$ is semiconcave with a  constant  satisfying 
$\sup_{i\in \cI}[u^n_i]_{2,+} \le Cn$.
\end{Theorem}

\begin{proof}

It is well understood  that the solution $u^0$ to a weakly coupled system with convex Hamiltonians is Lipschitz continuous under (H.\ref{assum:regularity})--(H.\ref{assum:lambda_lip}) (see \cite{briani2012}), and is semiconcave  if  the coefficients enjoy higher regularity (see the second case  in the proof of Proposition \ref{prop:semiconcave_ob}). 
We now use an inductive argument to estimate the regularity of the iterates $(u^n)_{n\in \N}$.

It has been shown in Proposition \ref{prop:iter_conv} that $(u^n)_{n\in \N}$ are  bounded uniformly in $n$. Now suppose $u^{n-1}$ is Lipschitz continuous, and (H.\ref{assum:M_lip}) holds. Then we can deduce from Proposition \ref{prop:ob_Lip} that
\bb\l{eq:lip_est}
\sup_{i\in \cI}[u^n_i]_1\le \max\bigg(C',\sup_{i\in \cI}[\cM_iu^{n-1}]_1\bigg)\le  \max\bigg(C',\sup_{i\in \cI}[u^{n-1}_i]_1+C\bigg),
\ee
where  $C'$ is a constant independent of $n$, and $C$ is the constant in (H.\ref{assum:M_lip}). An inductive argument enables us to conclude the desired estimate $[u^n_i]_{1}=\cO(n)$ for all $i$.

Moreover, by further assuming (H.\ref{assum:M_concave})  and the assumptions of Proposition \ref{prop:semiconcave_ob}, we can obtain  for all $i\in \cI$ the following estimate: for all $i\in \cI$,
\bb\l{eq:concave_est}
[u^n_i]_{2,+}\le \max 
\bigg(C'\sup_{i\in \cI} |u^n_i|_1, \sup_{i\in \cI}[\cM_iu^{n-1}_i]_{2,+}\!\bigg)\le \max \bigg(C'\sup_{i\in \cI} |u^n_i|_1, \sup_{i\in \cI}[u^{n-1}_i]_{2,+}+C\bigg), 
\ee
where  $C'$ is a constant independent of $n$, and $C$ is the constant in (H.\ref{assum:M_concave}).
Then, by using the previous Lipschitz estimates of $(u^n)_{n\in \N}$, we conclude from \eqref{eq:concave_est} that $[u^n_i]_{2,+}=\cO(n)$ for all $i$.
\end{proof}

\begin{Remark}\l{rmk:const_k}
Suppose (H.\ref{assum:M_lip}) and (H.\ref{assum:M_concave})  hold with $C=0$ (e.g.~the intervention operator $\cM$ is of the form \eqref{eq:M_common}), then the estimates \eqref{eq:lip_est} and \eqref{eq:concave_est} hold with $C=0$. Thus one can show inductively that for any $n\ge 0$, the Lipschitz constant $[u^n]_1$ and the semiconcavity constant $[u^n]_{2,+}$ of the iterate $u^n$ are   uniformly bounded in terms of $n$, which along with Proposition \ref{prop:iter_conv}, imply that the solution to HJBQVI \eqref{eq:qvi_impulse} is Lipschitz continuous and semiconcave.  As we shall see in Remark \ref{rml:penalty_const_k}, this observation enables us to  improve the convergence rate of the penalty approximation by a log factor.
\end{Remark}

\subsection{Regularization of penalized equations}\l{sec:regularization_penalty}
 In this section, we shall propose a sequence of auxiliary problems to the penalized equation \eqref{eq:penalty} with a fixed parameter $\rho>0$, which is 
 similar to the regularization of the QVI \eqref{eq:qvi_impulse} discussed in  Section \ref{sec:regularize}. These auxiliary problems will  serve as an important tool for quantifying convergence orders of the penalty approximations.
 
 More precisely, for any given penalty parameter $\rho>0$, we shall consider the following sequence of auxiliary problems: 
let $u^{\rho,0}=u^0$ be the solution to \eqref{eq:iter_0}, and for each $n\ge1$, given $u^{\rho,n-1}$, let $u^{\rho,n}$ be the solution to the following equations:
\bb\l{eq:penalty_n}
\sup_{\alpha\in\cA_i}\mathcal{L}^{\alpha}_i(x,u^{\rho,n}(x),Du^{\rho,n}_i(x),D^2u^{\rho,n}_i(x))+\rho(u^{\rho,n}_i-\cM_iu^{\rho,n-1})^+(x)=0, \q i\in \cI.
\ee
The above iterates   $(u^{\rho,n})_{n\ge 0}$ can be  equivalently expressed as $u^{\rho,n} = Q^\rho u^{\rho,n-1}$ for all $n\in \N$ with an operator $Q^\rho:[C^0_1(\R^d)]^M\to [C^0_1(\R^d)]^M$ defined as follows:  for any given $u$, $Q^\rho u$ is defined as the unique solution to the following  equations:
\bb\l{eq:Q_rho}
\sup_{\alpha\in\cA_i}\mathcal{L}^{\alpha}_i(x,Q^\rho u(x),D(Q^\rho u)_i(x),D^2(Q^\rho u)_i(x))+\rho((Q^\rho u)_i-\cM_iu)^+(x)=0, \q i\in \cI.
\ee

We now present some important properties of the operator $Q^\rho$.  Suppose (H.\ref{assum:regularity}) and (H.\ref{assum:M_lip}) hold, then for any given $u\in [C^0_1(\R^d)]^M$, we see $\cM u\in [C^0_1(\R^d)]^M$, from which one can establish   the comparison principle of \eqref{eq:Q_rho} and the well-posedness of \eqref{eq:Q_rho} in the class of bounded continuous functions.

The following lemma strengthens the comparison principle by  indicating that $Q^\rho$ is monotone and  Lipschitz continuous with constant 1, which is essential  for the error estimates in Section \ref{sec:rate}. The proof is included in Appendix \ref{appendix}.

\begin{Lemma}\l{lemma:Qrho_lip}
Suppose  (H.\ref{assum:regularity}) and (H.\ref{assum:M_lip}) hold. Then for any $u,v\in [C^0_1(\R^d)]^M$, we have that
$$\sup_{i\in\cI}\big|\big((Q^\rho u)_i-(Q^\rho v)_i\big)^+\big|_0\le\sup_{i\in\cI} \left|(u_i-v_i)^+\right|_0.$$
Consequently, if $u\le v$, then $Q^\rho u\le Q^\rho v$.
\end{Lemma}

Finally, a straightforward modification of the  doubling arguments for Lemma \ref{lemma:Qrho_lip} enables us to show that under the assumptions (H.\ref{assum:regularity}), (H.\ref{assum:lambda_lip})   and (H.\ref{assum:M_lip}), the solution $Q^\rho u$ to \eqref{eq:Q_rho} is in fact Lipschitz continuous provided that $u$ is Lipschitz continuous, which subsequently implies
the iterates $(u^{\rho, n})_{n\ge 0}$ are well-defined functions in $[C^0_1(\R^d)]^M$. We omit the  proof of these Lipschitz estimates, by pointing out that the analysis for the obstacle part is exactly the same as those for Lemma \ref{lemma:Qrho_lip}, and referring the reader to \cite{briani2012} for a discussion on the HJB part.

We then proceed to study the convergence of the iterates $(u^{\rho, n})_{n\ge 0}$. 
The next lemma shows the sequence $(u^{\rho,n})_{n\in \N}$ is monotone and uniformly bounded.

\begin{Lemma}
Suppose (H.\ref{assum:regularity})--(H.\ref{assum:lambda_lip})   and (H.\ref{assum:M_lip}) hold, and  $\rho>0$ is a fixed penalty parameter. Then the iterates $(u^{\rho,n})_{n\ge 0}$ are monotonically decreasing and uniformly bounded in terms of $n$.
\end{Lemma}
\begin{proof}
Note that the comparison principle of  \eqref{eq:Q_rho} yields  $u^{\rho,0}\ge u^{\rho,1}$,  which together with the monotonicity of $Q^\rho$ leads to  $u^{\rho,n-1}\ge u^{\rho,n}$ for all $n\ge 1$. 
Now we show by induction that $u^{\rho,n}\ge u^\rho$ for all $n\ge 0$, where $u^\rho$ is the solution to \eqref{eq:penalty}.
The statement holds clearly for $n=0$. Suppose for some $n\in \N$, we have $u^{\rho,n-1}\ge u^\rho$, then Lemma \ref{Lemma:cM} (2) implies $\cM_iu^{\rho,n-1}\ge \cM_i u^\rho$ for all $i\in \cI$, and hence $\rho(u^\rho-\cM_iu^{\rho})^+\ge \rho(u^\rho-\cM_iu^{\rho,n-1})^+$, which implies $u^\rho$ is a subsolution of the equation for $u^{\rho,n}$. Consequently, we obtain from the comparison principle that $u^{\rho,n}\ge  u^\rho$.
\end{proof}

The next theorem presents the convergence of $(u^{\rho, n})_{n\ge 0}$ to the solution of \eqref{eq:penalty}.

\begin{Theorem}\l{thm:conv_Q_rho}
Suppose  (H.\ref{assum:regularity})--(H.\ref{assum:lambda_lip}) and (H.\ref{assum:M_lip}) hold. 
Then for any given $\rho\ge 0$,  the  iterates $(u^{\rho,n})_{n\ge 0}$  converge monotonically from above to the solution $u^\rho$  of \eqref{eq:penalty} as $n\to \infty$.
\end{Theorem}
\begin{proof}
With the comparison principle of \eqref{eq:penalty} (Proposition \ref{prop:penalty_comparison}) in mind, it remains to show   the component-wise relaxed half-limit  $u^{*,\rho}$ (resp.~$u^{\rho}_*)$ of $(u^{\rho,n})_{n\ge 0}$  is a subsolution (resp.~supersolution) to \eqref{eq:penalty}. 
For notational simplicity, we shall omit the dependence on $\rho$ in the subsequent analysis  if no confusion can occur.

 Let $x\in \R^d$, $i\in \cI$ and $(p,X)\in J^{2,+}{u}^*_i(x)$. It follows from \cite[Lemma~6.1]{crandall1992} that there exist   $(x^n,p^n,X^n)_{n\in \N}$ such that $(p^n,X^n)\in J^{2,+}{u}^{\rho, n}_i(x^n)$ for each $n$ and $(x^n,u^{\rho, n}_i(x^n),p^n,X^n)\to (x,{u}^*_i(x),p,X)$ as $n\to\infty$. Then we have for all $n\in \N$ that 
\bb\l{eq:un_sub}
\sup_{\alpha\in\cA_i}\mathcal{L}^{\alpha}_i(x^n,u^{\rho,n}(x^n),p^n,X^n)+\rho(u^{\rho,n}_i-\cM_iu^{\rho,n-1})^+(x^n)\le 0.
\ee

Note that by Lemma \ref{Lemma:cM} (3), we have
\begin{align*}
\liminf_{n\to \infty} (u^{\rho,n}_i-\cM_iu^{\rho,n-1})^+(x^n)&\ge \liminf_{n\to \infty} (u^{\rho,n}_i-\cM_iu^{\rho,n-1})(x^n)
\ge (u^*_i-\cM_iu^{*})(x),
\end{align*}
which implies $\liminf_{n\to \infty}(u^{\rho,n}_i-\cM_iu^{\rho,n-1})^+(x^n)\ge (u^*_i-\cM_iu^{*})^+(x)$. Thus, for any given $\a \in \cA_i$, we can take  limit inferior in \eqref{eq:un_sub} and obtain from the inequality $\liminf_{n\to \infty}-d_{ij}^\a(x^\rho)u^\rho_j(x^\rho)\ge  -d_{ij}^\a(x)u^*_j(x)$ (see \eqref{eq:liminf_conv}) that
\begin{align*}
&\mathcal{L}^{\alpha}_i(x,u^{*}(x),p,X)+\rho(u^{*}_i-\cM_iu^{*})^+(x)\\
&\le\liminf_{n\to\infty} \mathcal{L}^{\alpha}_i(x^n,u^{\rho,n}(x^n),p^n,X^n)+\rho(u^{*}_i-\cM_iu^{*})^+(x)\le 0.
\end{align*}
Then taking the supremum over $\a\in \cA_i$ gives us the desired result.

We then turn to study ${u}_*$ by fixing $x\in \R^d$, $i\in \cI$ and $(p,X)\in J^{2,-}({u}_*)_i(x)$. Let $(x^n,p^n,X^n)_{n\in \N}$ be a sequence  such that $(p^n,X^n)\in J^{2,-}{u}^{\rho,n}_i(x^n)$ for each $n$ and $(x^n,{u}^{\rho,n}_i(x^n),p^n,X^n)\to (x,({u}_*)_i(x),p,X)$ as $n\to\infty$. Then for all $n\in \N$, the supersolution property of $u^{\rho,n}$ implies that
\begin{align*}
&\mathcal{L}^{\alpha_n}_i(x^n,u^{\rho,n}(x^n),p^n,X^n)\\
&=\sup_{\alpha\in\cA_i}\mathcal{L}^{\alpha}_i(x^n,u^{\rho,n}(x^n),p^n,X^n)
\ge -\rho(u^{\rho,n}_i-\cM_iu^{\rho,n-1})^+(x^n),
\end{align*}
for some $\a^n\in \cA_i$.
Then by taking limit superior as $n\to \infty$ on both sides of the above inequality and using similar arguments as  \eqref{eq:liminf_conv2}, we obtain that
\begin{align*}
&\sup_{\alpha\in\cA_i}\mathcal{L}^{\alpha}_i(x,(u_*)_i(x),p,X)\\
&\ge \limsup_{n\to \infty}-\rho(u^{\rho,n}_i-\cM_iu^{\rho,n-1})^+(x^n)\ge -\rho\limsup_{n\to \infty}(u^{\rho,n}_i-\cM_iu^{\rho,n-1})^+(x^n).
\end{align*}
Now it remains to show 
$$m\coloneqq \limsup_{n\to \infty}(u^{\rho,n}_i-\cM_iu^{\rho,n-1})^+(x^n)\le ((u_*)_i-\cM_iu_*)^+(x).$$
We shall assume without loss of generality that $m>0$. Then  by extracting a subsequence, we can further assume 
$u^{\rho,n}_i(x^n)>\cM_iu^{\rho,n-1}(x^n)$ for $n$, and $\lim_{n\to \infty}(u^{\rho,n}_i(x^n)-\cM_iu^{\rho,n-1}(x^n))=m$.
These properties along with Lemma  \ref{Lemma:cM} (3)  yield
$$
m=\limsup_{n\to \infty}(u^{\rho,n}_i(x^n)-\cM_iu^{\rho,n-1}(x^n))\le ({u}_*)_i(x)-\cM_i({u}_*)(x)\le ((u_*)_i-\cM_iu_*)^+(x),
$$
which finishes the proof of the statement that ${u}_*$ is a supersolution of \eqref{eq:penalty}.
\end{proof}

\subsection{Convergence rates of value functions}\l{sec:rate}
In this section, we shall exploit the regularization procedures discussed in Sections \ref{sec:regularize} and \ref{sec:regularization_penalty} to 
estimate the convergence rates of the penalty approximation, depending on the regularity of the coefficients. 

Let us first recall the penalty errors  for the classical obstacle problem, which have been analyzed in \cite{jakobsen2006,reisinger2018} and play an important role in our error estimates. 
To avoid confusion with the solutions to  \eqref{eq:qvi_impulse} and \eqref{eq:penalty}, we shall denote by $v$ the solution to  the obstacle problem \eqref{eq:iter_ob},
 and by $v^\rho$  the solution to the following penalized equation with a given parameter $\rho\ge 0$:
\bb\l{eq:iter_ob_penalty}
\sup_{\alpha\in\cA_i}\mathcal{L}^{\alpha}_i(x,v^\rho(x),Dv^\rho_i(x),D^2v^\rho_i(x))+\rho(v^\rho_i-\Psi_i)^+(x)=0, \q i\in \cI.
\ee

\begin{Proposition}\l{prop:err_ob}
For any given penalty parameter $\rho>0$, let $v$ and $v^\rho$ be the solution to the obstacle problem \eqref{eq:iter_ob} and  the penalized equation  \eqref{eq:iter_ob_penalty}, respectively. Suppose  (H.\ref{assum:regularity}) holds and   $\Psi_i\in C^0_1(\R^d)$  for all $i\in \cI$. 
Then there exists a constant $C$, independent of $(\Psi_i)_{i\in\cI}$, such that 
\bb\l{eq:half}
0\le v^\rho_i(x)-v_i(x)\le C\left(\sup_{i\in \cI}|\Psi_i|_1\right)\rho^{-1/2}, \q x\in\R^d, i\in \cI.
\ee
 If, in addition, $\Psi_i$  is semiconcave  for all $i\in\cI$, then we have
\bb\l{eq:1st}
0\le v^\rho_i(x)-v_i(x)\le C\sup_{i\in\cI}\bigg(|\Psi_i|_1+[\Psi_i]_{2,+}\bigg)/\rho,  \q x\in\R^d, i\in \cI.
\ee 
\end{Proposition}

\begin{proof}
The statement extends the results for scalar HJB equations studied in \cite{jakobsen2006}, and can be established by using similar arguments. The main step is to observe that for any given constant $C_\rho$ satisfying $C_\rho\ge \rho (v^\rho_i-\Psi_{i})^+$ for all $i\in \cI$, $v^\rho-C_\rho/\rho$ is a subsolution to \eqref{eq:iter_ob}, which implies $v^\rho-v\le C_\rho/\rho$. Then if we suppose that  $\Psi\in [C^2(\R^d)]^M$, one can deduce that 
there exists a constant $C>0$, independent of $\rho$ and $(\Psi_i)_{i\in \cI}$, such that the  upper bound 
$C_\rho \le C\sup_i\big(|\Psi_i|_1+|(D^2\Psi_i)^+|_0)$ holds for all $\rho\ge 0$, which enables us to conclude  \eqref{eq:1st} for smooth obstacles. Finally,  we  can regularize a general nonsmooth obstacle with mollifiers, and balance the approximation errors to obtain the desired error estimates \eqref{eq:half} and \eqref{eq:1st}.
\end{proof}

We then present the following elementary lemma, which extends \cite[Lemma 6.1]{bonnans2007} to polynomials with higher degrees.  The proof follows from a straightforward computation, which is included  in  Appendix \ref{appendix} for completeness.
\begin{Lemma}\l{lemma:min}
For any given $\a>0$, $\mu\in (0,1)$ and $\gamma\in \N$, consider the function  $\phi^\a:(0,\infty)\to \R$, $\phi^\a(x)=\a x^\gamma+\mu^x$. Then there exists a constant $C>0$, depending only on $\gamma$ and $\mu$, such that 
$$
m^\a \coloneqq \min_{n\in \N}\phi^\a(n)\le C\a(-\log \a)^\gamma, \q \textnormal{as $\a\to 0$.}
$$
\end{Lemma}

Now we are ready to state the main result of this paper, which gives an upper bound of the penalization error $u^\rho-u$.
\begin{Theorem}\l{thm:error_impulse}
Let $u$ and $u^\rho$ solve the QVI \eqref{eq:qvi_impulse} and  the penalized problem  \eqref{eq:penalty}, respectively. If (H.\ref{assum:regularity})--(H.\ref{assum:lambda_lip}) and (H.\ref{assum:M_lip}) hold, then for all  large enough penalty parameter $\rho$, we have
\bb\l{eq:penalty_half}
0\le u^\rho_i(x)-u_i(x)\le C(\log \rho)^2\rho^{-1/2}, \q x\in\R^d,\,  i\in \cI.
\ee
If we further assume (H.\ref{assum:regularity_concave}) and (H.\ref{assum:M_concave}) hold, and the constant $\lambda_0$ in (H.\ref{assum:regularity}) is sufficiently large, then 
\bb\l{eq:penalty_1st}
0\le u^\rho_i(x)-u_i(x)\le C(\log \rho)^2/\rho, \q x\in\R^d, \, i\in \cI,
\ee
for some constant  $C$  independent of the  parameter $\rho$.
\end{Theorem}
\begin{proof}
For notational simplicity, in the subsequent analysis, we shall denote by $C$  a generic constant, which is independent of the iterate index $n$ and the penalty parameter $\rho$, and may take a different value at each occurrence.

The monotone convergence (see Theorem \ref{thm:conv_penalty}) of $(u^\rho)_{\rho\ge 0}$ implies that $ u^\rho_i-u_i\ge 0$ for any given $\rho\ge 0$, hence it remains to establish an upper bound of $u^\rho-u$. Note that we have 
$$
u^\rho_i-u_i=(u^{\rho}_i-u^{\rho,n}_i)+(u^{\rho,n}_i-u^{n}_i)+(u_i^n-u_i),  \q i\in \cI, n\ge 0,
$$
where $u^{\rho,n}_i$ and $u^{n}_i$ solve \eqref{eq:penalty_n} and \eqref{eq:iter_n}, respectively. Since  $u^{\rho,n}\ge u^\rho$ for any $n\in \N$ (see Theorem \ref{thm:conv_Q_rho}), we obtain from  Proposition  \ref{prop:iter_conv} that
\bb\l{eq:error_bdd}
u^\rho_i-u_i\le (u^{\rho,n}_i-u^{n}_i)+(u_i^n-u_i)\le | u^{\rho,n}-u^n|_0+C\mu^n,  \q n\in \N,
\ee
for some constants $\mu\in (0,1)$ and $C> 0$.

We now estimate the term $| u^{\rho,n}-u^n|_0$. Since the operator $Q^\rho$ is Lipschitz continuous with constant 1 (see Lemma \ref{lemma:Qrho_lip}), it holds for all $n\in \N$ that
\bb\l{eq:u_rhon-u_n}
| u^{\rho,n}-u^n|_0\le | Q^\rho u^{\rho,n-1}-Q^\rho u^{n-1}|_0+| Q^\rho u^{n-1}-u^{n}|_0\le |  u^{\rho,n-1}- u^{n-1}|_0+| Q^\rho u^{n-1}-u^{n}|_0.
\ee
Then by letting the obstacle $\Psi_{i}=\cM_i u^{n-1}$ for all $i\in\cI$ in \eqref{eq:iter_ob} and using Proposition \ref{prop:err_ob}, we can obtain an upper bound of the last term in \eqref{eq:u_rhon-u_n} depending on the regularity of the iterates $(u^{n})_{n\ge 0}$.

In particular, under the assumptions (H.\ref{assum:regularity})--(H.\ref{assum:lambda_lip})  and (H.\ref{assum:M_lip}), we know from Theorem \ref{thm:iter_regularity} that $u^n$ is Lipschitz continuous with constant $|u^n_i|_1\le Cn$ for all $i\in \cI$ and $n\in \N$. Then by using the estimates \eqref{eq:half}, \eqref{eq:error_bdd} and \eqref{eq:u_rhon-u_n}, we get $u^\rho_i-u_i\le C(n^2\rho^{-1/2}+\mu^n)$ for all $n\in \N$, from which we can conclude \eqref{eq:penalty_half} by applying Lemma \ref{lemma:min} with $\a=\rho^{-1/2}$ and $\gamma=2$.

Similarly, by further assuming (H.\ref{assum:regularity_concave}), (H.\ref{assum:M_concave}), and the constant $\lambda_0$ in (H.\ref{assum:regularity}) is sufficiently large, we obtain from Theorem \ref{thm:iter_regularity} that   $|u^n_i|_1+[u^n_i]_{2,+}\le Cn$ for all $n$, which implies 
$u^\rho_i-u_i\le C(n^2/\rho+\mu^n)$ for all $n\in \N$, and subsequently leads to the desired estimate \eqref{eq:penalty_1st}.
\end{proof}
\begin{Remark}\l{rml:penalty_const_k}
As pointed out in Remark \ref{rmk:const_k},  in the case where the intervention operator satisfies (H.\ref{assum:M_lip}) and (H.\ref{assum:M_concave}) with $C=0$ (e.g.~$\cM_i$ is of the form \eqref{eq:M_common}), we know the iterates $(u^n)_{n\in \N}$ are  uniformly Lipschitz continuous  and uniformly semiconcave with respect to $n$. Therefore, by following the above arguments, we can improve the estimates \eqref{eq:penalty_half} and \eqref{eq:penalty_1st} to $\cO((\log \rho)\rho^{-1/2})$ and $\cO((\log \rho)\rho^{-1})$, respectively.
\end{Remark}

\subsection{Approximation of action regions and optimal impulse controls}\l{sec:region_ctrl}

In this section, we  propose  convergent approximations to the action regions and optimal control strategies of the HJBQVI \eqref{eq:qvi_impulse} based on the penalized equations. Since in general an optimal continuous control strategy may not exist due to the nonsmoothness of value functions, we shall focus on the approximation of optimal impulse controls.

Throughout this section, instead of specifying the precise convergence rates of the penalty schemes, which depend on the regularity of coefficients (see Theorem \ref{thm:error_impulse} and Remark \ref{rml:penalty_const_k}), we shall assume there exists a function $\om: (0,\infty)\to (0,\infty)$ such that $\om(\rho)\to 0$ as $\rho\to \infty$ and  
\bb\l{eq:modulus}
0\le u^\rho_i-u_i\le \om(\rho), \q i\in\cI, \, \rho>0.
\ee

For each $i\in \cI$, we shall approximate the action region of the $i$-th component $\cS_i=\{x\in \R^d\mid u_i(x)-\cM_iu(x)=0\}$ of \eqref{eq:qvi_impulse} by the following sets:
\bb\l{eq:S_rho}
\cS^\rho_i=\{x\in \R^d\mid |u^\rho_i(x)-\cM_iu^\rho(x)|\le \om(\rho)\}, \q \rho>0.
\ee
The next result shows that  $\cS^\rho_i$  converges to $\cS_i$  in  the Hausdorff metric.
\begin{Proposition}\l{prop:region}
Suppose   (H.\ref{assum:regularity}), (H.\ref{assum:impulse}) and the error estimate \eqref{eq:modulus} hold, and let 
$(\cS^\rho_i)_{i\in \cI}$ be the sets defined  in \eqref{eq:S_rho} for each $\rho>0$. Then  $\cS_i\subset \cS^\rho_i$  for all $i\in \cI$ and $\rho>0$. Moreover, it holds for any given compact set $\cK\subset \R^d$ that $\cS^\rho_i\cap \cK$ converges to $\cS_i\cap \cK$  in the Hausdorff metric as $\rho\to \infty$.

\end{Proposition}
\begin{proof}
The fact that $\cS_i\subset \cS^\rho_i$ follows directly from the estimate \eqref{eq:modulus} and the monotonicity of $\cM_i$ (see Lemma \ref{Lemma:cM} (2)). Hence  it remains to show that for any given compact set $\cK\subset \R^d$, we have $\lim_{\rho\to\infty} \sup_{y\in \cS_i^\rho\cap \cK}\inf_{x\in \cS_i\cap \cK}|y-x|=0$. Suppose it does not hold, then by passing to a subsequence, we know there exists $\eps>0$ and sequences  $y_n\in \cS_i^{\rho_n}\cap \cK$, $\rho_n\to \infty$, such that $y_n\to y^*\in \cK$ and $\inf_{x\in \cS_i\cap \cK}|y^*-x|\ge \eps$, i.e., $y^*\not \in \cS_i$. However, by using the continuity of the functions $u_i$ and $\cM_iu$ (see Lemma \ref{Lemma:cM} (3)), the definition of $\cS_i^\rho$, and the estimate \eqref{eq:modulus}, we can obtain: 
\begin{align*}
(u_i-\cM_iu)(y^*)=&(u_i-\cM_iu)(y^*)-(u_i-\cM_iu)(y_n)\\
&+(u_i-\cM_iu)(y_n)-(u^{\rho_n}_i-\cM_iu^{\rho_n})(y_n) +(u^{\rho_n}_i-\cM_iu^{\rho_n})(y_n)\\
\ge &(u_i-\cM_iu)(y^*)-(u_i-\cM_iu)(y_n)-\om(\rho_n)-\om(\rho_n)\to 0 
\end{align*}
as $n\to \infty$, which along with the fact $u_i\le \cM_iu$ on $\R^d$ implies $y^*\in \cS_i$, and hence leads to  a contradiction.
\end{proof}
\begin{Remark}\l{rmk:region}
It is essential to include the modulus of convergence $\om$ in the definition of $\cS^\rho_i$, since in general the naive approximation $\tilde{\cS}^\rho_i=\{x\in \R^d\mid u^\rho_i(x)-\cM_iu^\rho(x)=0\}$  does not give a convergent approximation to the action region $\cS_i$. For example,  let the vector $v=(v_l)_{l\in \{1,2\}}$ solve  the following discrete QVI:
$$
\max(v_1-b,\, v_1-(v_2+c))=0,\q \textnormal{and}\q \max(v_2-2b,\, v_2-(v_1+c))=0,
$$
where $b>c>0$. It is clear that the solution is given by $v_1=b$ and $v_2=b+c$, and the action region is the second index, i.e., $\cS=\{2\}$. However, for each $\rho>0$, one can directly verify that $v^\rho_1=b$ and $v^\rho_2=b+c+\f{b-c}{1+\rho}$ solve the penalized equation:
$$
v^\rho_1-b+\rho( v^\rho_1-v^\rho_2-c)^+=0,\q \textnormal{and}\q v^\rho_2-2b+ \rho(v^\rho_2-v^\rho_1-c)^+=0,
$$
which implies that $\tilde{\cS}^\rho=\emptyset$ for all $\rho$. 
\end{Remark}

Now we proceed to study optimal impulse control strategies. For any given $u\in [C^0(\R^d)]^M$, we denote by
$\cZ^u_i(x)\coloneqq \argmin_{z\in Z(x)}[u_i(\Gamma_i(x,z)) + K_i(x,z) ]$ the set of optimal impulse control strategies for all $i\in\cI$ and $x\in \cS_i$. 
The following result constructs a convergent approximation of $\cZ^u_i$, based on the set of impulse controls $\cZ^{u^\rho}_i(x)$, $x\in \cS^\rho_i$,   obtained by the penalized solution $u^\rho$.
\begin{Theorem}\l{thm:ctrl_approx}
Suppose the assumptions of Proposition \ref{prop:region} hold. Then for any $i\in \cI$, $x\in \cS_i$, and  sequence of impulse controls $(z^\rho)_{\rho>0}$ satisfying $z^\rho\in \cZ^{u^\rho}_i(x)$ for all $\rho$, we have 
$$\lim_{\rho\to \infty} \bar{d}_Z(z^\rho,\cZ^u_i(x))\coloneqq \lim_{\rho\to \infty} \inf\{d_{\boldsymbol{Z}}(z^\rho,z)\mid z\in \cZ^u_i(x)\}=0,$$ 
where $(\boldsymbol{Z},d_{\boldsymbol{Z}})$ is the metric space in (H.\ref{assum:impulse}). Consequently, if $\cZ^u_i(x)$ is a singleton, then $\cZ^{u^\rho}_i(x)$ converges to $\cZ^{u}_i(x)$ in the Hausdorff metric as $\rho\to \infty$.
\end{Theorem}
\begin{proof}
Suppose there exists $i\in \cI$, $x\in \cS_i$,  and a sequence  $(z^\rho)_{\rho>0}$ satisfying $z^\rho\in \cZ^{u^\rho}_i(x)$ and   $\bar{d}_Z(z^\rho,\cZ^u_i(x))\ge \eps>0$. Now let us consider the compact set $Z_\eps(x)=\{z\in Z(x)\mid \bar{d}_Z(z,\cZ^u_i(x))\ge \eps\}$, and pick ${z}_\eps\in Z_\eps(x)$ such that
$$
u_i(\Gamma_i(x,z_\eps)) + K_i(x,z_\eps)=\min_{z\in Z_\eps(x)}[ u_i(\Gamma_i(x,z)) + K_i(x,z) ].
$$
Since $Z_\eps(x)\cap \cZ^{u}_i(x)=\emptyset$, we can derive from the facts  $z_\eps\not \in \cZ^{u}_i(x)$ and $z^\rho\in Z_\eps(x)$ the following inequality:
\begin{align}\l{eq:contradiction}
\begin{split}
&[u_i(\Gamma_i(x,z^\rho)) + K_i(x,z^\rho)]-[u_i(\Gamma_i(x,\hat{z})) + K_i(x,\hat{z})]\\
\ge&\, [u_i(\Gamma_i(x,z_\eps)) + K_i(x,z_\eps)]-[u_i(\Gamma_i(x,\hat{z})) + K_i(x,\hat{z})]\coloneqq c_0>0,
\end{split}
\end{align}
for some $\hat{z}\in \cZ^{u}_i(x)$. On the other hand, we obtain from the estimate \eqref{eq:modulus} that
\begin{align*}
&[u_i(\Gamma_i(x,z^\rho)) + K_i(x,z^\rho)]-[u_i(\Gamma_i(x,\hat{z})) + K_i(x,\hat{z})]\\
=&\, [u_i(\Gamma_i(x,z^\rho))-u^\rho_i(\Gamma_i(x,z^\rho))]+[u^\rho_i(\Gamma_i(x,z^\rho)) + K_i(x,z^\rho)]-[u_i(\Gamma_i(x,\hat{z})) + K_i(x,\hat{z})]\\
\le  &\,\cM_i u^\rho(x)-\cM_i u(x)\le |u^\rho_i-u_i|_0\le \om(\rho),
\end{align*}
 which contradicts  \eqref{eq:contradiction} by passing $\rho\to \infty$, and finishes the proof.
\end{proof}
\begin{Remark}
We refer the reader to \cite{guo2009,pham2009} and references therein, where the  uniqueness of a pointwise optimal impulse strategy has been established for various practical impulse control problems by  exploiting the regularity of the value functions and the  structure of the intervention operator. Then in Theorem \ref{thm:ctrl_approx} the convergence of the approximate controls follows.
\end{Remark}

\section{Extension to some HJBQVIs with signed  costs}\l{sec:switching}
In this section, we extend the penalty schemes to  a class of QVIs  with possibly negative impulse costs which arise from optimal switching problems.
{In this setting, the controller has two mechanisms of affecting the \emph{regime switching process} $I$ in \eqref{eq:u_i}, namely through their continuous control process $\alpha$ acting on its Markov transition matrix, as well as directly and immediately by exercising an impulse control to change the regime, the latter at the expense of a positive impulse cost or benefitting from a negative impulse cost.}
We shall propose an efficient, alternative penalty scheme by taking advantage of the finiteness of the  set of switching controls, and extend the  convergence analysis in Section \ref{sec:error estimate} to estimate the penalization error. 


More precisely, we consider the following system of HJBQVIs: for each $i\in \cI=\{1,\ldots, M\}$ and $x\in \R^d$,
\begin{align}\l{eq:qvi_s}
\begin{split}
&F_i(x,u,Du_i,D^2u_i)\\
&\coloneqq \max\Big\{
\sup_{\alpha\in\cA_i}\mathcal{L}^{\alpha}_i(x,u(x),Du_i(x),D^2u_i(x)), \   
(u_i-\mathcal{M}_i u)(x)\Big\}=0, 
\end{split}
\end{align}
where the linear operator $\cL^\a_i$ is defined as in \eqref{eq:defLi}, 
and the intervention operator $\cM_i$ is given by
\bb
\label{eq:M_switching}
(\mathcal{M}_i u)(x) = \min_{j\in \cI^{-i} } \{ u_j(x) + k_{ij}(x) \},\q x\in \R^d.
\ee

By enlarging the state space $\R^d$ into the product space $\R^d\t\cI$, we can treat \eqref{eq:M_switching} as a special case of \eqref{eq:M_impulse} with $Z_i(x)=\cI^{-i}$, $\Gamma_i(x,z)=(x,z)$ and $K_i(x,z)=k_{iz}(x)$ for all $(x,i)\in \R^d\t \cI$. Consequently, if the switching costs $k_{ij}$ are strictly positive, i.e., $k_{ij}\ge \kappa_0>0$, we can directly apply  the penalty scheme  \eqref{eq:penalty} to solve \eqref{eq:qvi_s}, and deduce from Theorem \ref{thm:error_impulse} the rate of convergence in terms of the  parameter $\rho$. 

As we shall see shortly, the structure of the operator $\cM_i$ and the finiteness of  the set of impulse controls allow us to   consider signed switching costs $(k_{ij})_j$ taking both positive and negative values.

Now we   introduce an alternative  penalty scheme for solving  \eqref{eq:qvi_s}. For any given penalty parameter $\rho\ge 0$, we consider the following system of  penalized equations: for all $i \in\cI$, $x\in \R^d$,
\begin{align}\l{eq:penalty_s}
\begin{split}
&F^\rho_i(x,u,Du_i,D^2u_i)\\
&\coloneqq\sup_{\alpha\in\cA_i}\mathcal{L}^{\alpha}_i(x,u^\rho(x),Du^\rho_i(x),D^2u^\rho_i(x))+\rho\sum_{j\in\cI^{-i}}(u^\rho_i-u^\rho_j-k_{ij})^+(x)=0.
\end{split}
\end{align}
Unlike   \eqref{eq:penalty}, the above penalty scheme 
makes use of the finiteness of the  set $\cI^{-i}$, and  performs 
penalization on each component of the system, which leads to easily implementable and efficient iterative schemes for the penalized equations without taking the pointwise maximum over all switching components (see \cite{reisinger2018qvi}).

\begin{Remark}\l{rmk:pointwise}
The  penalty scheme \eqref{eq:penalty_s} can be extended to the general intervention operator \eqref{eq:M_impulse}, for which we introduce the following penalty term:
$$
\rho \int_{Z_i(x)}\bigg(u_i(x)- u_i(\Gamma_i(x,z)) - K_i(x,z) \bigg)^+\, \nu(dz), \q \rho\ge 0,
$$
where $\nu$ is a given finite measure supported on the set $\cup_{i,x}Z_i(x)$ (see \cite{kharroubi2010}).
\end{Remark}

In the remaining part of this section,  we shall discuss how  to extend the convergence analysis in the previous sections  to study penalty schemes for \eqref{eq:qvi_s}  with possibly negative switching costs.
We shall focus on the scheme \eqref{eq:penalty_s}, but the same analysis extends naturally to the scheme \eqref{eq:penalty}.
 More precisely, we shall replace (H.\ref{assum:impulse}) by the following condition on the switching costs:
\begin{Assumption}\l{assum:sc}
There exist constants $C\ge 0$ and $\kappa_0>0$ such that for all  $j\not =i,l \in \cI$, we have $k_{ii}\equiv 0$,
\begin{align}\l{eq:triangle}
k_{ij}(x)+k_{jl}(x)&- k_{il}(x)\ge \kappa_0>0, \q  x\in \R^d,
\end{align}
and the following regularity estimates: $|k_{ij}|_1\le C$, and  $k_{ij}$ is semiconcave with  constant $C$ around any point $x\in \R^d$ with $k_{ij}(x)<\kappa_0$.
\end{Assumption}

The allowance of negative switching costs clearly complicates the assumptions on the switching costs, which is worth a detailed discussion.
The   triangular condition \eqref{eq:triangle} is similar to the assumption used in \cite{pham2009,lundstrom2014}, which means that it is less expensive  to switch directly  from regime $i$ to $l$ than in two steps via an intermediate regime $j$. It also implies $k_{ij}+k_{ji}\ge\kappa_0>0$ for all $j\not=i$, which prevents arbitrage opportunities that one can gain a positive profit  by instantaneously switching back and forth. This further leads to the ``no  loop condition" introduced by \cite{ishii1991switching},
%
which together with (H.\ref{assum:regularity}) enables us to conclude a  comparison principle of \eqref{eq:qvi_s} by using similar arguments as those for \cite[Theorem~2.1]{lundstrom2014}, and consequently the uniqueness of viscosity solutions to \eqref{eq:qvi_s} in the class of bounded continuous functions.

The  Lipschitz continuity and semiconcavity assumptions in (H.\ref{assum:sc}) are similar to those in \cite{lundstrom2014}, which 
ensure the existence of  a strict subsolution to \eqref{eq:qvi_s} (see Proposition \ref{prop:subsolution}). However, we remark that, instead of requiring the switching costs to be semiconcave on $\R^d$ as in \cite{lundstrom2014}, we only impose the semiconcavity condition around the  points at which the   costs are close to or less than zero, hence  no additional regularity is required if we are in the classical context of strictly positive switching costs.  

The following proposition explicitly constructs a strict subsolution to \eqref{eq:qvi_s}, which is crucial to the well-posedness
of \eqref{eq:qvi_s} and \eqref{eq:penalty_s}, but also the error estimates of the penalty approximations (cf.~Propositions \ref{prop:penalty_comparison} and \ref{prop:iter_conv}).
\begin{Proposition}\l{prop:subsolution}
Suppose (H.\ref{assum:regularity}) and (H.\ref{assum:sc}) hold. Then there exists a constant $C>0$, such that for any $\eps\in (0,\kappa_0)$, the function $w\in [C^0_1(\R^d)]^M$ defined as 
\bb\l{eq:strict_sub}
w_i=-\tilde{k}_i-C, \q \tilde{k}_i=\min\left\{\min_{j\in\cI^{-i}}(k_{ji}-\eps),0\right\}, \q i\in \cI,
\ee
is a strict subsolution to \eqref{eq:qvi_s} and \eqref{eq:penalty_s} for any $\rho\ge 0$, i.e., $F_i(x,u,Du_i,D^2u_i)\le -\min(\eps,\kappa_0-\eps)$ and $F^\rho_i(x,u,Du_i,D^2u_i)\le -\min(\eps,\kappa_0-\eps)$ in the viscosity sense.

\end{Proposition}
\begin{proof}
For any given $\eps\in (0,\kappa_0)$, we first verify  $w_i- \cM_iw\le -\min(\eps,\kappa_0-\eps)$. Note that 
$$
w_i-w_j-k_{ij}=-\tilde{k}_i+\tilde{k}_j-k_{ij}=-\tilde{k}_i+\min\left\{\min_{l\in\cI^{-j}}(k_{lj}-\eps),0\right\}-k_{ij}, \q \fa j\in \cI^{-i}.
$$
Now if $\tilde{k}_i=0$, we can pick $l=i$ and deduce that $w_i-w_j-k_{ij}\le k_{ij}-\eps-k_{ij}=-\eps$. Otherwise, if $\tilde{k}_i=k_{mi}-\eps$ for some $m\in \cI^{-i}$, then we shall separate the discussions into two cases. If $m=j$, then \eqref{eq:triangle} and the facts that $k_{ii}=0$, $\tilde{k}_j\le 0$ imply that $w_i-w_j-k_{ij}\le -(k_{ji}-\eps)-k_{ij}\le -(\kappa_0-\eps)$. On the other hand, if $m\not =j$, by setting $l=m$, we obtain that 
$$w_i-w_j-k_{ij}\le -(k_{mi}-\eps)+(k_{mj}-\eps)-k_{ij}\le -\kappa_0,$$
which completes the proof of the desired statement by taking the maximum over all $j\in \cI^{-i}$.

Note that for any given $i\in \cI$ and $x\in\R^d$, $\tilde{k}_i(x)$ is  defined by taking the minimum over the indices $j\in \cI^{-i}$ such that $k_{ji}(x)\le \eps<\kappa_0$, hence by using (H.\ref{assum:sc}) one can show $\tilde{k}_i(x)$ is semiconcave with some constant $C\ge 0$ around $x$. Therefore, we can infer for each $i\in\cI$ that $\tilde{k}_i$ is Lipschitz continuous and semiconcave in $\R^d$. Hence there exists a sequence of smooth functions $(\tilde{k}^\eps)_{\eps>0}$ such that  $D(-\tilde{k}^\eps)$ is bounded and $D^2(-\tilde{k}^\eps)$ is   bounded below uniformly in terms of $\eps$, and $\tilde{k}^\eps$ uniformly converges to $\tilde{k}$   as $\eps\to 0$.
Then by using  the   boundedness of coefficients and the stability of subsolutions, we deduce that there exists a constant $C'$ such that for all $i\in \cI$ and $x\in \R^d$, we have
$\sup_{\alpha\in\cA_i}\mathcal{L}^{\alpha}_i(x,-\tilde{k},D(-\tilde{k}_i),D^2(-\tilde{k}_i))\le  C'$
in the viscosity sense. Hence, for any constant  $C$ such that
$C\ge (C'+\min(\eps,\kappa_0-\eps))/\lambda_0$, we can conclude that $w\in [C^0_1(\R^d)]^M$ is a strict subsolution to \eqref{eq:qvi_s} and \eqref{eq:penalty_s} for any $\rho\ge 0$.
\end{proof}

With the strict subsolution in hand, we can establish the existence of solutions to \eqref{eq:penalty_s} (cf.~Proposition \ref{prop:penalty_comparison}), the monotone convergence of \eqref{eq:penalty_s} (cf.~Theorem \ref{thm:conv_penalty}), and also the error estimate of the iterated optimal stopping approximation of \eqref{eq:qvi_s} (cf.~Proposition \ref{prop:iter_conv}). 

Moreover, we can easily see that (H.\ref{assum:M_lip}) and (H.\ref{assum:M_concave}) hold provided that the switching costs enjoy sufficient regularity. In fact, it is clear that if $u\in [C^0_1(\R^d)]^M$ and $[k_{ij}]_1\le C$ for all $j\in \cI^{-i}$, then $\cM_i u\in C^0_1(\R^d)$ satisfies $[\cM_iu]_1\le \sup_{j\in \cI^{-i}}([u_j]_1+[k_{ij}]_1)$. If $u_i$ and $k_{ij}$ are semiconcave in $\R^d$ for all $i,j\in \cI^{-i}$, then $\cM_i u$ is semiconcave in $\R^d$ with constant $[\cM_iu]_{2,+}\le \sup_{j\in \cI^{-i}}([u_j]_{2,+}+[k_{ij}]_{2,+})$. Therefore,  we can obtain as a direct consequence of Theorem \ref{thm:iter_regularity} that the iterates $(u^n)_{n\in \N}$ are Lipschitz continuous with constant $\cO(n)$ if the switching costs are Lipschitz continuous, and they are semiconcave with constant $\cO(n)$ if the switching costs are semiconcave.

Finally, by assuming the obstacles $(\Psi_i)_{i\in \cI}$ in \eqref{eq:iter_ob} are of the form $\Psi_i=\min_{j\in \cI^{-i}}\Psi_{ij}$ for all $i\in \cI$, we can generalize Proposition \ref{prop:err_ob} to study the following penalty approximation to the classical obstacle problem \eqref{eq:iter_ob}:
$$
\sup_{\alpha\in\cA_i}\mathcal{L}^{\alpha}_i(x,v^\rho(x),Dv^\rho_i(x),D^2v^\rho_i(x))+\rho\sum_{j\in \cI^{-i}}(v^\rho_i-\Psi_{ij})^+(x)=0, \q i\in \cI.
$$
and obtain exactly the same error estimates \eqref{eq:half} and \eqref{eq:1st}. 

Now we are ready to conclude the following analogue of Theorem \ref{thm:error_impulse}, which gives the convergence rate of \eqref{eq:penalty_s} to  \eqref{eq:qvi_s} with respect to the penalty parameter.

\begin{Theorem}\l{thm:error_switching}
Let $u$ and $u^\rho$ solve the QVI \eqref{eq:qvi_s} and  the penalized problem  \eqref{eq:penalty_s}, respectively. If (H.\ref{assum:regularity}), (H.\ref{assum:lambda_lip}) and (H.\ref{assum:sc}) hold, then for all  large enough penalty parameter $\rho$, we have
$$
0\le u^\rho_i(x)-u_i(x)\le C(\log \rho)^2\rho^{-1/2}, \q x\in\R^d,\,  i\in \cI.
$$
If we further assume (H.\ref{assum:regularity_concave}) holds, the constant $\lambda_0$ in (H.\ref{assum:regularity}) is sufficiently large, and $(k_{ij})_{i,j\in \cI}$ are semiconcave in $\R^d$,  then we have
$$
0\le u^\rho_i(x)-u_i(x)\le C(\log \rho)^2/\rho, \q x\in\R^d, \, i\in \cI,
$$
for some constant  $C$,  independent of the  parameter $\rho$ and the number of switching components $M$.
\end{Theorem}

\section{Discretization and policy iteration for penalized equations}\l{sec:discrete}
In this section, we shall discuss briefly how to construct convergent discretizations for  the penalized equations, and propose a globally convergent iterative method to solve the discretized equation based on policy iteration.

Let us start with the discretization of the penalized equation \eqref{eq:penalty}  with a fixed penalty parameter $\rho>0$. 
We shall denote by $\{x_l\}_{l}=h\Z^d$  a  uniform spatial grid  on $\R^d$ with mesh size $h$,  by $u^\rho_{i,l}$  the discrete approximation to $u^\rho_i$ at the point $x_l$, and by $Z_{i,l}$ the set of impulse controls at the point $x_l$. 

It is standard to show  that, by  using  monotone discretizations (e.g.~the semi-Lagrangian scheme in \cite{debrahant2012})  for the differential operators and  multilinear interpolations for the intervention operator (see \cite{azimazadeh2018,reisinger2018}),  one can derive  the following approximation to  \eqref{eq:penalty}:  for all $i\in \cI$,
\begin{align}\l{eq:scheme_im}
\begin{split}
\sup_{\alpha\in\cA_i} &\bigg[\sum_{m\in \Z^d}\theta^\a_{i,l,m}(u^\rho_{i,l}-u^\rho_{i,m})+c^\a_{i,l}u^\rho_{i,l}-\sum_{j\in \cI^{-i}} d^\a_{ij,l}u^\rho_{j,l}-\ell^\a_{i,l}\\
&+\rho\bigg(u^\rho_{i,l}-\inf_{z\in Z_{i,l}}\bigg[\sum_{m\in \Z^d}\gamma^z_{i,l,m}u^\rho_{i,m}+K^z_{i,l}\bigg]\bigg)^+ \bigg]=0,\q  l\in \Z^d,
\end{split}
\end{align}
with some  coefficients $\theta^\a_{i,l,m}\ge 0$, $0\le \gamma^z_{i,l,m}\le 1$ and $\sum_{m\in \Z^d}\gamma^z_{i,l,m}=1$ for all $l,m\in \Z^d$, $\a\in \cA_i$ and $z\in Z_{i,l}$. Under (H.\ref{assum:regularity}) and (H.\ref{assum:impulse}), it is  straightforward to  show that the above scheme is   monotone and consistent with the penalized equation \eqref{eq:penalty} as $h$ tends to zero,  which enables us to conclude from \cite[Proposition 3.3]{briani2012} that the  numerical solution of \eqref{eq:scheme_im} converges to the solution of  \eqref{eq:penalty} as $h\to 0$. 
Moreover, one can deduce by similar arguments as those in \cite{azimazadeh2018} that the numerical solution converges to the solution of the QVI \eqref{eq:qvi_impulse} when $1/\rho$ and $h$ tend to zero simultaneously.

Now we proceed to demonstrate the global convergence of policy iteration for solving \eqref{eq:scheme_im}. We shall first enlarge the control space and reformulate \eqref{eq:scheme_im} into an HJB equation in a countably infinite space. 
Note that by introducing the set $\cB=\{0,1\}$ and using the fact $\sum_{m\in \Z^d}\gamma^z_{i,l,m}=1$, one can  rearrange the terms of \eqref{eq:scheme_im} and obtain that: for all $(i,l)\in \cI\t \Z^d$,
\begin{align*}
\sup_{(\alpha,\b,z)\in\cA_i\t \cB\t Z_{i,l}}
& \bigg[\bigg(\sum_{m\not=l}(\theta^\a_{i,l,m}+\b\rho\gamma^z_{i,l,m})+c^\a_{i,l}\bigg)u^\rho_{i,l}&\\
&-\sum_{m\not=l}(\theta^\a_{i,l,m}+\b\rho\gamma^z_{i,l,m})u^\rho_{i,m}
-\sum_{j\in \cI^{-i}} d^\a_{ij,l}u^\rho_{j,l}-\ell^\a_{i,l}-\b\rho K^z_{i,l}\bigg]=0,
\end{align*}
which can be equivalently expressed in the following compact form: 
\bb\l{eq:hjb}
\sup_{\om\in \cA} \bigg(\tilde{A}(\om)\u^\rho-\tilde{b}(\om)\bigg)=0,
\ee
where $\u^\rho=(u^\rho_{i,l})_{(i,l)\in\cI\t \Z^d}$, $\cA=(\cA_i\t \cB\t Z_{i,l})^{\cI\t \Z^d}$, and  
for any given $\om=(\a_{i,l},\b_{i,l},z_{i,l})_{(i,l)\in\cI\t \Z^d}\in\cA$, 
$\tilde{A}(\om)=(\tilde{a}_{(i,l),(i',l')}(\om))_{(i,l),(i',l')\in\cI\t \Z^{d}}$ 
is the following ``infinite" matrix (see \cite{bokanowski2009}): 
\bb
\tilde{a}_{(i,l),(i',l')}(\om)=\begin{cases} \sum_{m\not=l}(\theta^{\a_{i,l}}_{i,l,m}+\b_{i,l}\rho\gamma^{z_{i,l}}_{i,l,m})+c^{\a_{i,l}}_{i,l}, &i'=i, l'=l,\\
-(\theta^{\a_{i,l}}_{i,l,l'}+{\b_{i,l}}\rho\gamma^{z_{i,l}}_{i,l,l'}), & i'=i, l'\not =l,\\
-d^{\a_{i,l}}_{ii',l}, & i'\not =i, l=l.
\end{cases}
\ee
 Now we can apply the classical policy iteration to solve \eqref{eq:hjb}, or equivalently \eqref{eq:scheme_im}: let $\om^{(0)}$ be a given initial control value, for all $k\ge 0$, define $(\u^{\rho,(k)}, \om^{(k+1)})$ as follows:
\bb\l{eq:pi_p}
\tilde{A}(\om^{(k)})\u^{\rho,(k)}-\tilde{b}(\om^{(k)})=0,\q\q \om^{(k+1)}\in \argmax_{\om\in \cA} (\tilde{A}(\om)\u^{\rho,(k)}-\tilde{b}(\om)),
\ee
where the maximization is performed component-wise. Such maximization operation is well-defined  under (H.\ref{assum:regularity}) and (H.\ref{assum:impulse}), due to the fact that the control set $\cA_i\t \cB\t Z_{i,l}$ is compact and the coefficients $\tilde{A}$ and $\tilde{b}$ are continuous in $\om$. 

The next theorem establishes the monotone convergence of  $(\u^{\rho,(k)})_{k\ge 0}$ for any initial guess $\om^{(0)}$, which extends the result in  \cite{azimzadeh2016weakly} to  weakly coupled systems in an infinite dimensional setting.

\begin{Theorem}\l{thm:global}
Suppose (H.\ref{assum:regularity}) and (H.\ref{assum:impulse}) hold. Then for any initial control value $\om^{(0)}$, the iterates  $(\u^{\rho,(k)})_{k\ge 0}$ are well-defined, and converge pointwise to the unique solution of \eqref{eq:hjb}, or equivalently \eqref{eq:scheme_im}, as $k\to \infty$. Moreover, we have $\u^{\rho,(k)}\ge \u^{\rho,(k+1)}$ for all $k\ge 0$.
\end{Theorem}
\begin{proof}
The statement is an analogue of Proposition B.1 in \cite{bokanowski2009}, where the monotone convergence of policy iteration has been proved for concave HJB equations. Note  \eqref{eq:coeff_mono} implies that for each $\om\in \cA$ and $(i,l)\in\cI\t \Z^d$, 
$$
\tilde{a}_{(i,l),(i,l)}(\om)\ge \sum_{(i',l')\not=(i,l))}|\tilde{a}_{(i,l),(i',l')}(\om)|+\lambda_0,
$$ 
which gives the monotonicity of $\tilde{A}$, i.e., for any given $\om\in \cA$, if $\tilde{A}(\om)\u\ge 0$ and $\u$ is bounded, then $\u\ge 0$.   Moreover, 
the boundedness of coefficients leads to the uniform boundedness of the iterates $(u^{(k)})_{k\ge 0}$ and the fact that  $\sup_{\om\in \cA}(\textrm{Card}\{(i',l')\mid \tilde{a}_{(i,l),(i',l')}(\om)\not=0\})<\infty$ for each $(i,l)\in\cI\t \Z^d$. Therefore, even though the control set in \eqref{eq:hjb} varies for each component $(i,l)$, it is straightforward to adapt the arguments for \cite[Proposition B.1]{bokanowski2009} and establish the desired convergence result.
\end{proof}
\begin{Remark}\l{rmk:direct}
Theorem \ref{thm:global} establishes one of the major advantages of penalty schemes over the  direct control scheme studied in \cite{azimzadeh2016weakly,chancelier2007}, which  applies policy iteration to solve a direct discretization of QVI \eqref{eq:qvi_impulse}. Such a scheme in general is not well-defined due to the  possible singularity of the matrix iterates caused by the non-strict monotonicity  of $u_i-\cM_i u$ in $u$. 
In fact, consider the  simple QVI $\max(u-g, u-\cM u)=0$ with  $\cM u\coloneqq u+c$ and $c>0$, 
whose solution is given by $u=g$ due to the fact that $u-\cM u=-c<0$. 
Suppose that we initialize policy iteration with the impulse control, then we need to solve $u-(u+c)=0$, which clearly admits no solution. More complicated examples can be constructed to show that the  direct control scheme  can fail at any intermediate iterate (see \cite{azimzadeh2016weakly}). 

\end{Remark}

\section{Numerical experiments}\l{sec:num}
In this section, we illustrate the theoretical findings and demonstrate the efficiency improvement of  the penalty schemes over the direct control scheme through numerical experiments. We shall present an infinite-horizon optimal switching problem and examine  the performance of penalty schemes with respect to the spatial mesh size and the penalty parameter.

 We first  introduce the following two-regime infinite-horizon optimal switching problem (see e.g.~\cite{pham2009,reisinger2018qvi}). Let $(\Om, \cF_t, \bP)$ be a filtered probability space and $\gamma=(\gamma_t)_{t\ge 0}$ be a  control process such that $\gamma_t=\sum_{k\ge 0}i_k1_{[\tau_k,\tau_{k+1})}(t)$, where  $(\tau_k)_{k\ge 0}$ is a non-decreasing sequence of stopping times representing the decision on ``when to switch", and 
for each $k\ge 0$, $i_k$ is an $\cF_{\tau_k}$-measurable random variable valued in the discrete space $\cI=\{1,2\}$, representing the   decision on ``where to switch". That is, the decision maker chooses  regime $i_k$ at the time $\tau_k$ for all $k\ge 0$. 

For any given switching control strategy $\gamma$, we consider the following controlled state  equation:
$$
dX^\gamma_t=(r+\nu(\gamma_t)(\mu-r))X^\gamma_t dt+\sigma\nu(\gamma_t) X^\gamma_t\,dW_t, \q t>0, \q X^\gamma_0=x,
$$
where $r,\mu,\sigma,x>0$ are given constants,  $(W_t)_{t>0}$ is a one-dimensional Brownian motion defined on $(\Om, \cF_t, \bP)$, and $\nu(i)=i-1$, $i\in \cI$. Then the objective function associated with the control strategy $\gamma$ is given by:
$$
J(x,\gamma)=\ex\bigg[\int_0^\infty e^{-rt}\ell(X^\gamma_t)\,dt-\sum_{k\ge 0}e^{-r\tau_{k+1}}c_{i_k,i_{k+1}}\bigg],
$$
where $\ell$ represents the running reward function and $c_{i,j}$ represents the switching cost from regime $i$ to $j$, $\fa i,j\in \cI$. For each $i\in \cI$, let $\bA^i$ be all  control strategies starting with regime $i$, i.e., $i_0=i$ and $\tau_0=0$. Then 
 the decision maker has the following value functions:
$$
u_i(x)=\sup_{\gamma\in \bA^i}J(x,\gamma), \q i\in \cI=\{1,2\}.
$$

Suppose   that the switching costs $c_{i,j}\equiv c>0$, $i\not =j$, then we can deduce from the   dynamic programming principle (see \cite{pham2009}) that the value functions $(u_1,u_2)$  satisfy the following system of quasi-variational inequalities: for all $i\in \cI$, $j\not =i$, $x\in (0,\infty)$,
\bb\l{eq:qvi}
\min\!\bigg[-\f{1}{2}\sigma^2\nu(i)^2x^2 D^2u_i(x)-(r+\nu(i)(\mu-r))xDu_i(x)+ru_i(x)-\ell(x), (u_i-u_j+c)(x)\bigg]\!\!=0.
\ee
Moreover,  even though \eqref{eq:qvi} involves a pointwise minimization instead of a pointwise maximization as  in \eqref{eq:qvi_s}, for any given penalty parameter $\rho>0$,  one can easily extend the scheme \eqref{eq:penalty_s} and derive the corresponding penalized equation for \eqref{eq:qvi}: for all $i\in \cI$, $j\not =i$, $x\in (0,\infty)$,
\bb\l{eq:qvi_p}
-\f{1}{2}\sigma^2\nu(i)^2x^2 D^2u^\rho_i(x)-(r+\nu(i)(\mu-r))xDu^\rho_i(x)+ru_i(x)-\ell(x)-\rho (u^\rho_j-c-u^\rho_i)^+(x)=0.
\ee
For our numerical experiments, we  set the parameters as  $c=1/8$, $\sigma = 0.2$, $\mu = 0.06$, $r = 0.02$ and choose a nonsmooth running reward function: $\ell(x)=0.5-|x-1|$ for $x\in [0.5,1.5]$ and $\ell(x)=0$ otherwise.

Now let $\rho>0$, $n\in \N$, and $\{x_l \} = \{lh\}_{l\in \N\cup\{0\}}$ be a uniform grid of $(0,\infty)$ with the mesh size $h=2^{-n}$. We shall derive a monotone discretization of the penalized equation \eqref{eq:qvi_p} by  employing the standard  (two-point) forward difference for the first derivates and (three-point) central difference for all second derivatives; see Section \ref{sec:discrete} and \cite{azimazadeh2018} for the convergence of the discretization as  $n,\rho\to \infty$.
We shall also localize the equation on the computational domain $(0,2)$ with homogenous Dirichlet boundary condition $u=0$ at $x = 2$, which leads to the following discrete equation for \eqref{eq:qvi_p}:
find $\u^\rho_N=(\u^\rho_{1,N},\u^\rho_{2,N})\in \R^{N}$ satisfying
 \begin{align}\l{eq:qvi_pd}
 \begin{split}
 & A\u^\rho_N-\vec{\ell}-\rho(b-M\u^\rho_N)^+\\
& \coloneqq \begin{pmatrix}B_1+rI_{N/2} & 0\\ 0 & B_2+rI_{N/2} \end{pmatrix}\u^\rho_N-\begin{pmatrix}\ell \\ \ell \end{pmatrix}-\rho\max(b-M\u^\rho_N,0)=0,
\end{split}
\end{align}
where $N=4/h=2^{n+2}$ is the total number of unknowns, $B_1,B_2\in \R^{N/2\t N/2}$ are  matrices resulting from  discretization of the differential operators, $\ell\in \R^{N/2}$ is a vector such that $\ell_k=\ell(x_{k-1})$ for all $k=1,\ldots, N/2$,
\[M=\begin{pmatrix} I_{N/2} & -I_{N/2} \\ -I_{N/2} & I_{N/2}\end{pmatrix}\]
is a matrix representation of the switching operator, and $b\in\R^N$ is a constant vector with value $-c$. Similarly, we can discretize \eqref{eq:qvi} for the direct control scheme: 
find $\u_N=(\u_{1,N},\u_{2,N})\in \R^{N}$ satisfying
\bb\l{eq:qvi_d}
\min(A\u_N-\vec{\ell}, M\u_N-b)=0. 
\ee

In the following, we shall discuss the implementation details for solving  \eqref{eq:qvi_pd} and  \eqref{eq:qvi_d} with policy iteration. 
The direct control scheme, which will serve as a benchmark for our penalized schemes, applies policy iteration to the discrete equation \eqref{eq:qvi_d} directly (see \cite{chancelier2007,azimzadeh2016weakly}). More precisely, 
 let $\om^{(0)}\in \{0,1\}^N$ be a given initial control value. Then, for all $k\ge 0$, we find  $(\u^{(k)}, \om^{(k+1)})\in \R^N\t  \{0,1\}^N$ such that 
\bb\l{eq:pi_d}
A^{(k)}\u^{(k)}-b^{(k)}=0,\q\q \om^{(k+1)}\in \argmin_{\om\in \{0,1\}}\bigg[ (1-\om)(A\u^{(k)}_N-\vec{\ell})+ \om (M\u^{(k)}_N-b)\bigg],
\ee
where the $i$th row of the matrix $A^{(k)}$ and the $i$-th component of the vector $b^{(k)}$ are determined by:
$$
A^{(k)}_i= (1-\om^{(k)}_i)A_i+\om^{(k)}_iM_i, \q b^{(k)}_i=(1-\om^{(k)}_i)\vec{\ell}_i+\om^{(k)}_i b, \q i=1,\ldots, N.
$$
The  iteration will be terminated once a desired tolerance is achieved, i.e., 
\bb\l{eq:tol}
\f{\|\u^{(k)}_N-\u^{(k-1)}_N\|}{\max(\|\u^{(k)}_N\|,\textrm{scale})} < \textrm{tol},
\ee
 where $\|\cdot\|$ denotes the sup-norm, and  the scale parameter is  chosen to guarantee that no unrealistic level of accuracy will be imposed if the solution is close to zero. On the other hand, the penalized scheme views  \eqref{eq:qvi_pd} with a given penalty parameter $\rho$ as a discrete HJB equation, and applies policy iteration \eqref{eq:pi_p} to solve it, which will be terminated by the same criterion \eqref{eq:tol}, with $(\u^{(k)}_N)_{k\ge 0}$  replaced by $(\u^{\rho,(k)}_N)_{k\ge 0}$.
We take $\textrm{tol}=10^{-9}$ and $\textrm{scale} = 1$ for all the experiments, and perform computations using \textsc{Matlab} R2018a on a 2.70GHz Intel Xeon E5-2680  processor. 

We reiterate that, compared with the global convergence of policy iteration \eqref{eq:pi_p} applied to the penalized equation \eqref{eq:qvi_pd}, policy iteration \eqref{eq:pi_d} applied to  \eqref{eq:qvi_d} in general is not well-defined for an arbitrary initial guess $\om^{(0)}$, as already observed in Remark \ref{rmk:direct} and \cite{azimzadeh2016weakly}. In fact, if we initialize \eqref{eq:pi_d} with $\om^{(0)}= \{1\}^N$, then we need to solve $M\u^{(0)}_N-b=0$, which has no solution due to the structure of the matrix $M$ and the fact $b=-c<0$. Therefore, we shall initialize policy iteration for 
\eqref{eq:qvi_pd} and \eqref{eq:qvi_d} with the continuation value, i.e., $\u^{(0)}_N=\u^{\rho,(0)}_N$ satisfying $A\u^{(0)}_N=\vec{\ell}$, which admits a solution since $A$ is a monotone matrix.

\begin{figure}[!ht]
    \centering
    \includegraphics[keepaspectratio=true, width=0.6\columnwidth ]{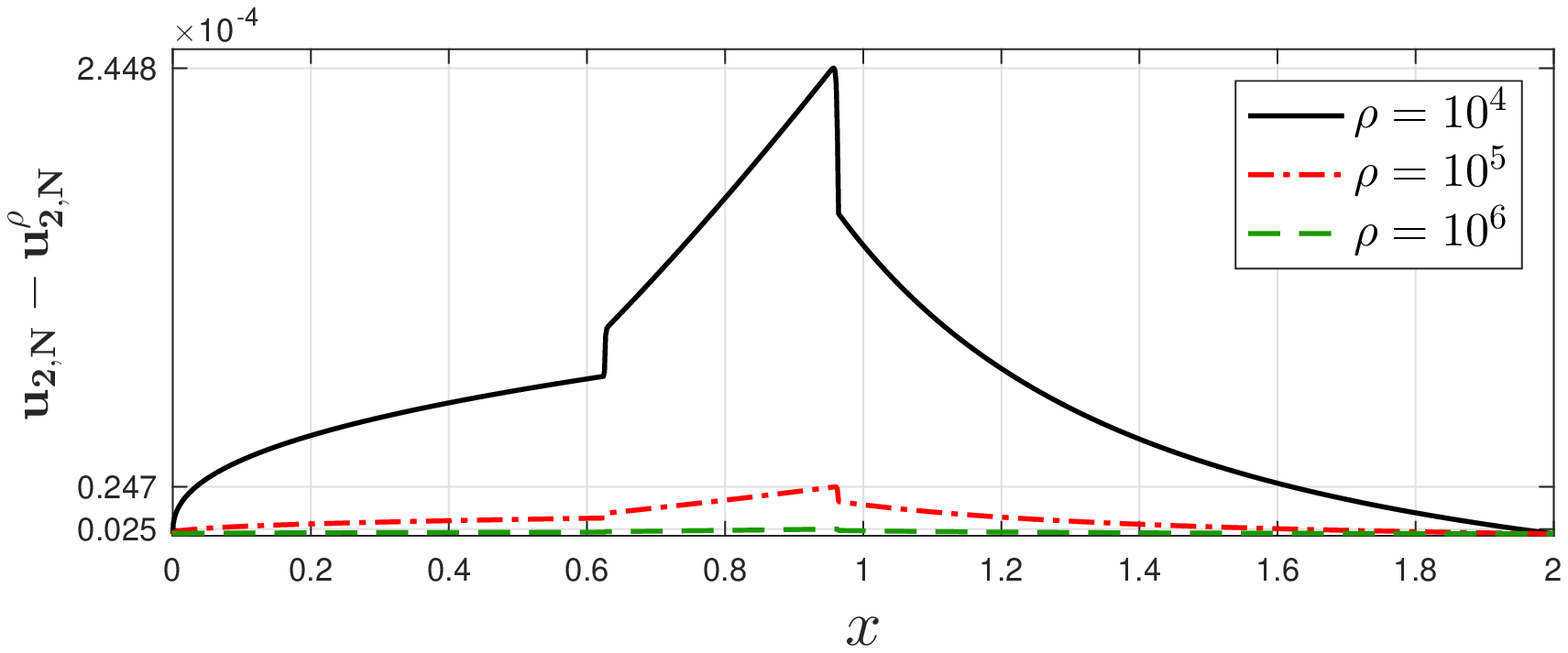}
        \includegraphics[keepaspectratio=true, width=0.6\columnwidth]{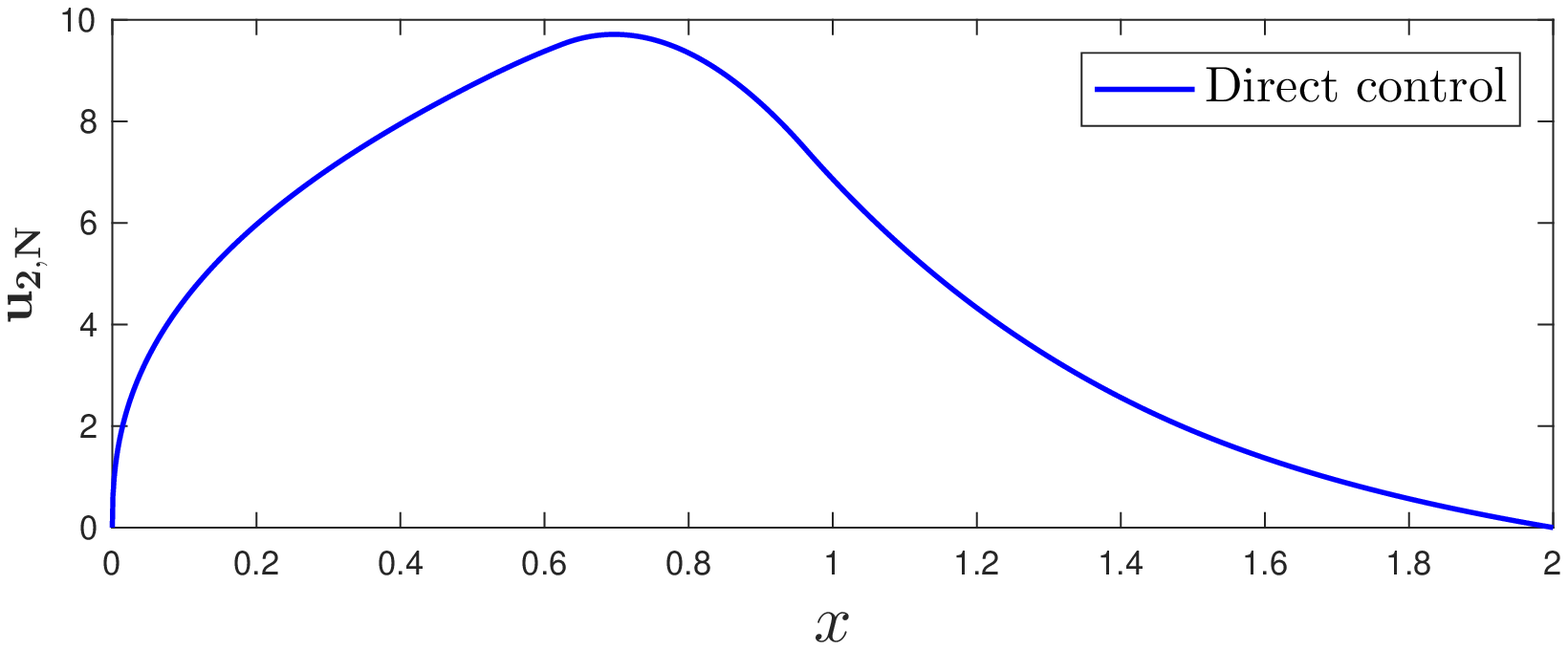}
    \caption{Numerical solutions of the value function $u_2$ obtained by the direct control scheme and the penalty schemes with different  penalty parameters ($N=65536$). Shown are: the difference $\u_{2,N}-\u^\rho_{2,N}$ of numerical solutions (top), and the  numerical solution $\u_{2,N}$ of the direct control scheme (bottom).}
    \label{fig:penalty_err}
\end{figure}

We start by  examining the convergence of the penalized schemes with respect to the penalty parameter and the mesh size. 
Figure \ref{fig:penalty_err} presents, for a fixed mesh size $h = 2^{-14}$ (the total number of unknowns is $N=65536$), the difference between the numerical solutions obtained by the direct control scheme and the penalty scheme with different  penalty parameters. It clearly indicates that, as the penalty parameter $\rho\to \infty$,  
the penalized solutions converge monotonically from below to the solution of the direct control scheme. 
Since the value function is sufficiently smooth (Figure \ref{fig:penalty_err}, bottom), 
we can also observe first order convergence of the penalization error (in the sup-norm) with respect to the penalty parameter $\rho$.

Table \ref{table:space_err} summarizes, for different mesh sizes, the numerical solutions of the direct control scheme and the penalty scheme with a fixed  parameter $\rho=10^5$.
It is interesting to observe that, for a fixed mesh size,  the spatial discretization errors of both the direct control scheme and the penalty scheme are of the same magnitude and converge to zero with first order as the mesh size tends to $0$. 
Moreover, the penalty parameter $\rho=10^5$ already leads to a negligible penalization error (compared to the  discretization error), which seems  to be stable with respect to different mesh sizes.


\begin{table}[!h]
 \renewcommand{\arraystretch}{1.05} 
\centering
\caption{
Results for the direct control scheme and the penalty scheme ($\rho=10^5$) with different mesh sizes.}
\label{table:space_err}
\begin{tabular}[t]{@{}rccc@{}}
\toprule
N & $16384$ & $32768$ & $65536$\\
\midrule
Direct control scheme \\
$\u_{1,N}(x=1)$ & 6.9339733 &  6.9330192 & 6.9325423\\
$|\u_{1,N}-\u_{1,N/2}|(x=1)$ &  &  $9.54\t 10^{-4}$ & $4.77\t 10^{-4} $ \\
Penalty scheme $(\rho=10^5)$ \\
$\u^{\rho}_{1,N}(x=1)$ & 6.9339645 &  6.9330100 & 6.9325330\\
$|\u^{\rho}_{1,N}-\u^{\rho}_{1,N/2}|(x=1)$ &  &  $9.54\t 10^{-4}$ & $4.77\t 10^{-4} $\\
\midrule
$\|\u_N-\u^{\rho}_N\|$ & $2.42\t 10^{-5}$  & $2.47\t 10^{-5}$  & $2.47\t 10^{-5}$ 
\\
\bottomrule
\end{tabular}
\end{table}%

We proceed to analyze the computational efficiency of the direct control scheme and the penalty scheme. Figure \ref{fig:efficiency} compares, for different mesh sizes and penalty parameters, the number of required policy iterations and the computational time of both schemes. 
One can observe clearly from Figure \ref{fig:efficiency}, left, that the number of required iterations  for the direct control scheme (the blue line) exhibits a linear growth  in  the size of the discrete system.   Moreover, 
 our experiments show that policy iteration applied to \eqref{eq:qvi_d} with fine meshes, i.e., $N\in \{131072,262144\}$, is not able to meet the desired accuracy within $10^5$  iterations, which suggests that  the direct control scheme may diverge for sufficiently fine meshes. On the other hand,  for penalty schemes with fixed penalty parameters 
(the green and black lines in Figure \ref{fig:efficiency}, left), the number of required iterations 
eventually stabilizes to a finite value for all  fine meshes, which is significantly less than the number of iterations for the direct control scheme.

One can further compare the overall runtime of the direct control scheme and the penalty scheme for solving discrete systems with different sizes $N$ (Figure \ref{fig:efficiency}, right). Note that for both methods, the computational time per iteration grows at a rate $O(N)$ due to the linear system solver. 
Hence, the total runtime of the direct control scheme increases at a rate $O(N^2)$ due to the  linear growth of the required iterations (the blue line), while the penalized scheme (with a fixed penalty parameter) achieves  a linear complexity in the computational time (the green and black lines), benefiting from  a mesh-independence property of policy iteration for penalized equations.  This suggests that the penalty schemes are significantly more efficient than the direct control scheme for solving large-scale discrete QVIs, as pointed out in \cite{azimzadeh2016weakly}.

\begin{figure}[!htb]
    \centering

    \includegraphics[width=0.496\textwidth,height=5.6cm]{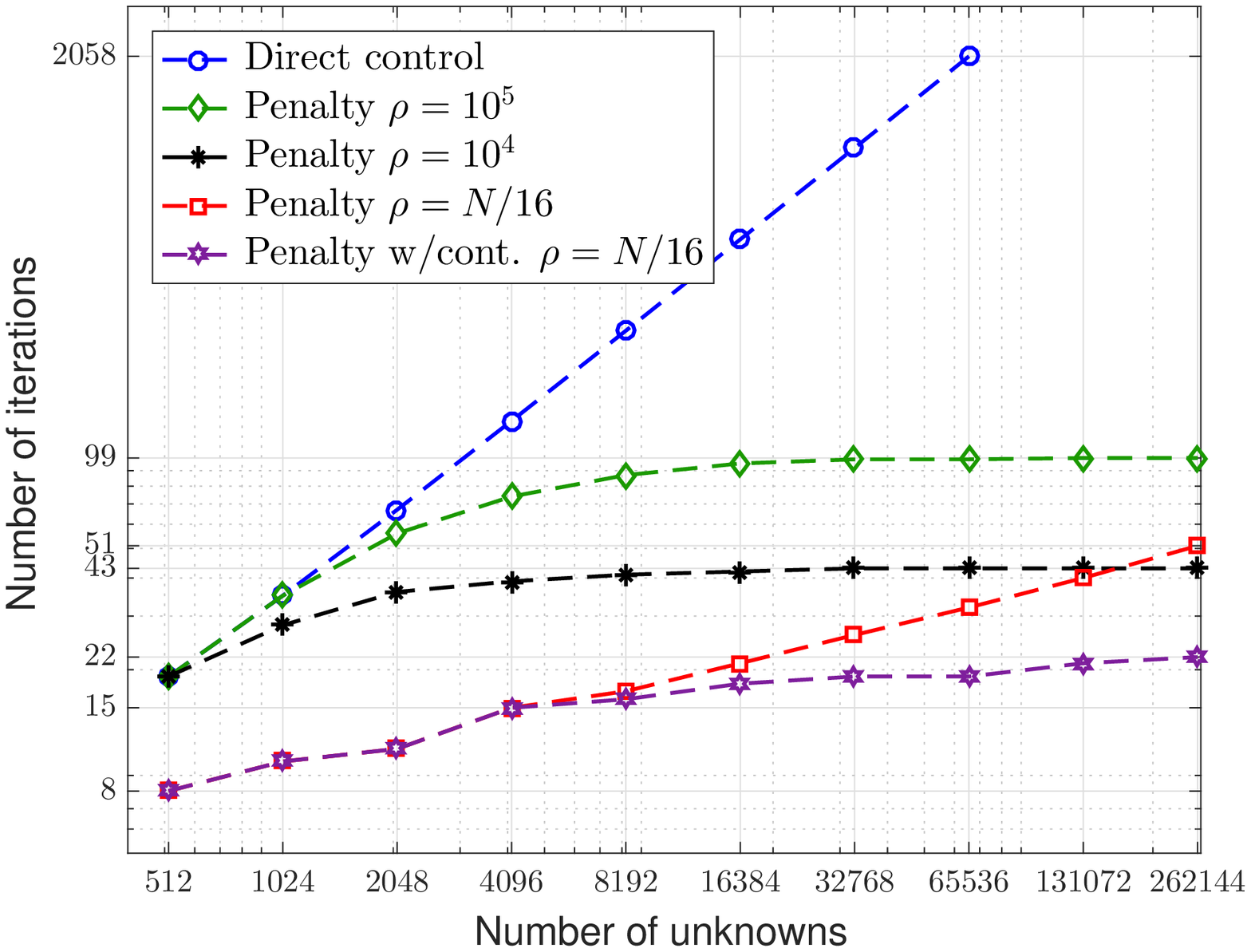} 
    \includegraphics[width=0.496\textwidth,height=5.6cm]{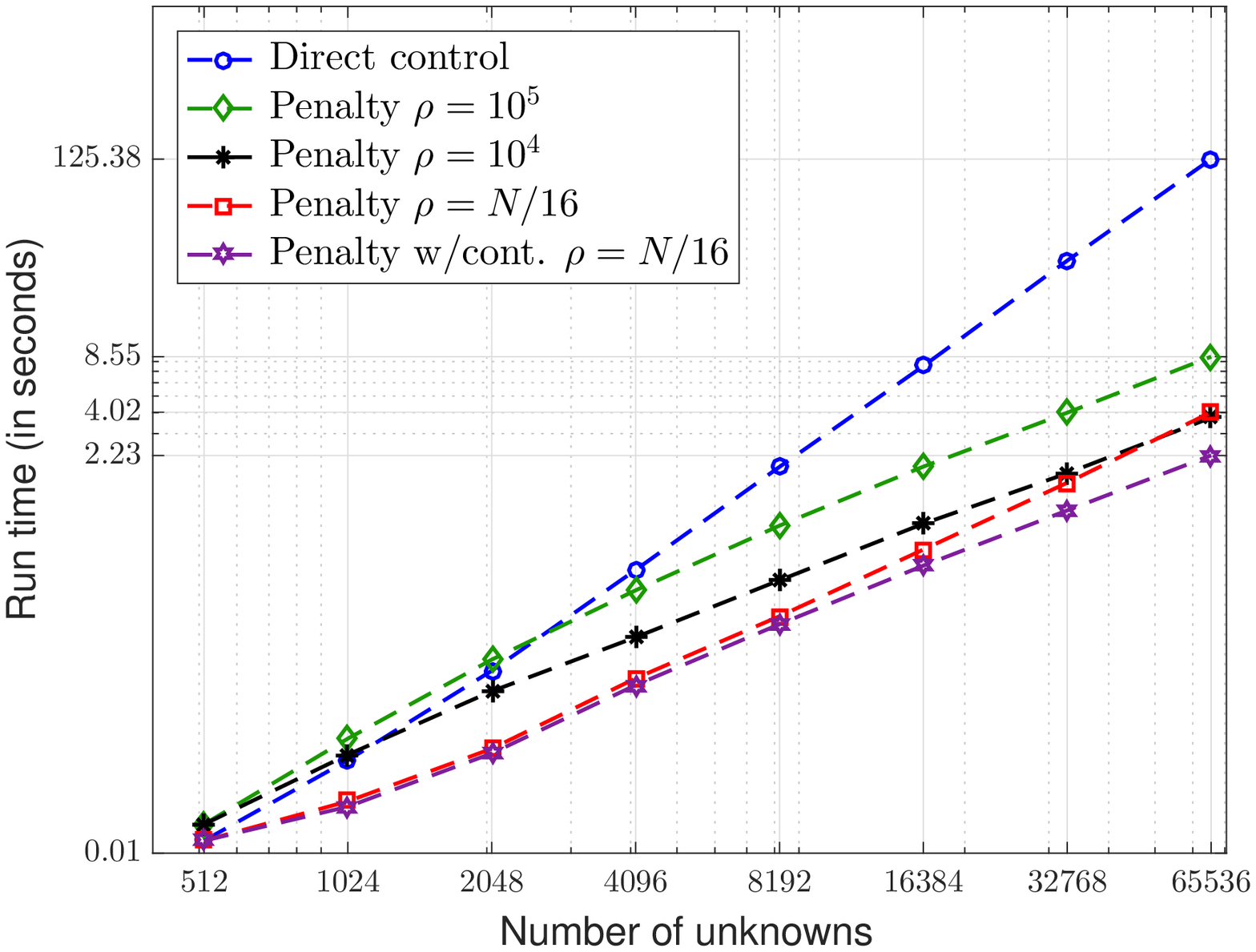} 
    \caption{Comparison of  the number of  iterations and the runtime for the direct control scheme and the penalty method with different  mesh sizes and penalty parameters (plotted in a log-log scale).}
    \label{fig:efficiency}
\end{figure}

In practice, instead of solving the penalized equation \eqref{eq:qvi_p} with a fixed penalty parameter $\rho$, we shall  construct a convergent approximation to  the solution of the QVI \eqref{eq:qvi} based on the penalized solutions, by 
letting $1/\rho$ and $h$ tend to zero simultaneously (see also \cite{azimzadeh2016weakly,azimazadeh2018}). The first order convergence of both the penalization error and the discretization error (see Figure \ref{fig:penalty_err} and Table \ref{table:space_err}) suggests us to take $\rho=CN$, where the constant $C=1/16$ was found to achieve the optimal balance between the penalization error and the discretization error. 
Moreover, as suggested in \cite{ito2003}, we can  combine the penalty method with a  continuation procedure in $\rho$ to further improve the algorithm's  efficiency.  In particular, given a discrete penalized equation \eqref{eq:qvi_pd} of size $N$, if the corresponding penalty parameter $\rho=N/16>200$,  we shall first solve a penalized equation \eqref{eq:qvi_pd} with the parameter $\rho=100$ by using the initialization $\u^{(0)}=A^{-1}\ell$, and then use the solution as the initialization for the algorithm with the desired parameter $\rho$.

Figure \ref{fig:efficiency} depicts the performance of the penalty scheme with the parameter $\rho=N/16$ (the red line) and  the penalty scheme with the  parameter $\rho=N/16$  and a continuation procedure (the purple line).
The increasing penalty parameter results in  an increasing number of iterations, but the growth rate is  much lower  than that of the discrete control scheme. A linear regression of the data without continuation procedure shows that the number of iterations  is of the magnitude $O(N^{0.3})$. 
Moreover, the continuation strategy effectively enhances the efficiency of the algorithm, and the number of iterations has only a mild dependence on the size of the system.

We finally remark that one can choose $\Delta t=O(h)$ and $1/\rho=O(h)$ to construct a convergent penalty approximation to solutions of parabolic HJBQVIs. It has been observed in practice (see Table 6.6 in \cite{azimzadeh2016weakly}) that the number of iterations for the penalty scheme remains stable with respect to  the mesh refinement, due to the fact that refining the mesh size in general produces a more accurate initial guess for policy iteration,
while  the direct control scheme  requires an increasing number of policy iterations per timestep as the mesh size tends to zero,  which 
leads to significantly more policy iterations  for high levels of refinement.

\section{Conclusions}
This paper develops a penalty approximation to systems of HJB quasi-variational inequalities (HJBQVIs) stemming from hybrid control problems involving impulse controls. We  established the monotone convergence of the penalty schemes and estimated the convergence orders, which subsequently led to  convergent approximations of action regions and optimal impulse controls. We further proved the monotone convergence of policy iteration for the penalized equations in an infinite dimensional setting. 
Numerical examples for infinite-horizon optimal switching problems are presented to illustrate the theoretical findings and to demonstrate the efficiency improvement of the penalty schemes over the classical direct control scheme.
\color{black}

To the best of our knowledge, this is the first paper which derives rigorous error estimates for  penalty approximations of  HJBQVIs, and proposes convergent approximations to action regions and optimal impulse controls. The penalty schemes and convergence results can be easily extended to nonlocal elliptic HJBQVIs arising from impulse control problems of jump-diffusion processes with regime switching. Natural next steps would be to extend the penalty approach to  parabolic HJBQVIs  as in \cite{seydel2009},  and to monotone systems with  bilateral obstacles arising from switching games \cite{ishii1990}.

\appendix

\section {Proofs of Lemma \ref{Lemma:cM} (3), Propositions \ref{prop:penalty_comparison} and \ref{prop:iter_conv}, and Lemmas \ref{lemma:Qrho_lip} and \ref{lemma:min}}\l{appendix}

\begin{proof}[Proof of Lemma \ref{Lemma:cM} (3)]
Let $x^\rho, x\in \R^d$ for all $\rho\in \N$ and $\lim_{\rho\to \infty}x^\rho= x$, we first establish that $\limsup_{\rho \to \infty} (\cM_i u^\rho)(x^\rho)\le (\cM_i u^*) u(x)$. For any $\eps>0$, there exists $z^\eps\in Z(x)$, such that $u^*_i(\Gamma_i(x,z^\eps)) + K_i(x,z^\eps)-\eps\le (\cM_i u^*)(x)$. Since $Z(x^{\rho})$ converges to $Z(x)$ in the Hausdorff metric, we can find $z^{\rho,\eps}\in Z(x^{\rho})$, such that $\lim_{\rho\to \infty}z^{\rho,\eps}=z^\eps$. Then  we conclude the desired result from the continuity of $\Gamma_i, K_i$ and the following inequality: for all $\eps>0$,
\begin{align*}
\limsup_{\rho \to \infty} (\cM_i u^\rho)(x^\rho)&\le \limsup_{\rho \to \infty}[u^\rho_i(\Gamma_i(x^\rho,z^{\rho,\eps})) + K_i(x^\rho,z^{\rho,\eps})]
\!\le u^*_i(\Gamma_i(x,z^{\eps})) + K_i(x,z^{\eps})\\
&\le (\cM_i u^*)(x)+\eps.
\end{align*}

We then show $(\cM_i u_* )(x)\le \liminf_{\rho \to \infty} (\cM_i u^\rho)(x^\rho)$. For any $\eps>0$ and $\rho\in \N$, there exists $z^{\rho,\eps}\in Z(x^{\rho})$ such that $u^\rho_i(\Gamma_i(x^\rho,z^{\rho,\eps})) + K_i(x^\rho,z^{\rho,\eps})-\eps\le (\cM_i u^\rho)(x^\rho)$. The fact that $Z(x^{\rho})$ is  convergent to the compact set $Z(x)$ implies that by passing to a subsequence, one can assume $(z^{\rho,\eps})_{\rho\in \N}$ is convergent to some $z^\eps\in Z(x)$. Then we have
\begin{align*}
\liminf_{\rho \to \infty} (\cM_i u^\rho)(x^\rho)&\ge\liminf_{\rho \to \infty}[ u^\rho_i(\Gamma_i(x^\rho,z^{\rho,\eps})) + K_i(x^\rho,z^{\rho,\eps})-\eps]\\
&\ge (u_*)_i(\Gamma_i(x,z^{\eps})) + K_i(x,z^{\eps})-\eps
\ge (\cM_i u_* )(x)-\eps,
\end{align*}
which completes the proof by letting $\eps\to 0$.
\end{proof}

\begin{proof}[Proof of Proposition \ref{prop:penalty_comparison}]
Let $u$ and $v$ be a bounded subsolution  and supersolution of \eqref{eq:penalty} with a fixed penalty parameter $\rho\ge 0$, respectively. We observe that for sufficiently large constant $C>0$, $w=-C$ is a subsolution to
$F^\rho_i(x,w,Dw_i,D^2w_i)\le -\kappa_0<0$, from which by using  the fact that  $F^\rho$ is convex in $u$, $Du$ and $D^2u$, we deduce that  $u_m\coloneqq (1-\f{1}{m})u+\f{1}{m}w$ is a  subsolution to $F^\rho_i(x,u_m,D(u_m)_i,D^2(u_m)_i)\le -\kappa_0/m$ for all $m\in \N$. Note that it suffices to show $u_m-v\le 0$ for all $m\in \N$, since one can deduce the desired comparison principle $u-v\le 0$ by letting $m\to \infty$.

Now suppose that there exists $m_0\in \N$ such that $M=\sup_{x\in\R^d,i\in\cI} ((u_{m_0})_i-v_i)(x)>0$, and consider for each $\eps>0$ the following quantity 
\bb
M_\eps=\sup_{x,y\in \R^d,i\in \cI} ((u_{m_0})_i(x)-v_i(y)-\f{1}{2\eps}|x-y|^2).
\ee
Then, by assuming without loss of generality that there exists an $i\in \cI$, independent of $\eps$, such that  the maximum is obtained at the index $i$ and the point $(x^\eps, y^\eps)$ (otherwise one can modify the test function with an additional penalty term), one can deduce from the standard arguments (see \cite{crandall1992}) that $\lim_{\eps\to 0}M_\eps=M$ and $\lim_{\eps\to 0}x^\eps=\lim_{\eps\to 0}y^\eps=x_0$ for some $x_0$. Thus by applying the maximum principle (\cite[Theorem 3.2]{crandall1992}),  we have for any given $\theta>1$  the matrices $X, Y\in \bS^d$ such that $(p_x, X)\in \bar{J}^{2,+}u_m(x^\eps)$ and $(-p_y,-Y)\in \bar{J}^{2,-}v(y^\eps)$, where 
$$
(p_x,p_y)=\f{1}{\eps}(x^\eps-y^\eps, y^\eps-x^\eps), \q\textnormal{and},\q\begin{pmatrix} X &0 \\ 0& Y  \end{pmatrix}\le \theta \f{1}{\eps}\begin{pmatrix} I &-I \\ -I& I  \end{pmatrix},
$$
from which, by using the   sub- and supersolution properties, we have
\begin{align}
\begin{split}
\sup_{\alpha\in\cA_i}&\mathcal{L}^{\alpha}_i(x^\eps,u_{m_0}(x^\eps),p_x,X)-\sup_{\alpha\in\cA_i}\mathcal{L}^{\alpha}_i(y^\eps,v(y^\eps),-p_y,-Y)\\
&+\rho((u_{m_0})_i-\cM_i u_{m_0})^+(x^\eps)-\rho(v_i-\cM_i v)^+(y^\eps) +\kappa_0/m_0\le 0.
\end{split}
\end{align}

Now we separate our discussions into two cases. Suppose  for all small enough $\eps$, we have
$$
\rho((u_{m_0})_i-\cM_i u_{m_0})^+(x^\eps)-\rho(v_i-\cM_i v)^+(y^\eps) \le -\kappa_0/m_0,
$$
which implies $(v_i-\cM_i v)(y^\eps)\ge 0$ and 
$$
((u_{m_0})_i-\cM_i u_{m_0})(x^\eps)-(v_i-\cM_i v)(y^\eps) \le -\kappa_0/(\rho m_0).
$$
Then by rearranging the terms in the above inequality and using the definition of $M_\eps$, we have
\begin{align*}
M&=\lim_{\eps\to 0}M_\eps=\lim_{\eps\to 0}\big[(u_{m_0})_i(x^\eps)-v_i(y^\eps)-|x^\eps-y^\eps|^2/(2\eps)\big]\\
&\le \limsup_{\eps\to 0}(\cM_i u_{m_0})(x^\eps)-\liminf_{\eps\to 0}(\cM_i v)(y^\eps)-\liminf_{\eps\to 0}|x^\eps-y^\eps|^2/(2\eps)-\kappa_0/(\rho m_0)\\
&\le (\cM_i u_{m_0})(x_0)-(\cM_i v)(x_0)-\kappa_0/(\rho m_0)\le M-\kappa_0/(\rho m_0),
\end{align*}
where we have used Lemma \ref{Lemma:cM} (3) and the fact that $u_{m_0}$ and $v$ are upper- and lower-semicontinuous, respectively. This clearly contradicts to the fact that $\kappa_0/(\rho m_0)>0$.

On the other hand, suppose for all small enough $\eps$, we have
$$\sup_{\alpha\in\cA_i}\mathcal{L}^{\alpha}_i(y^\eps,v(y^\eps),-p_y,-Y)-\sup_{\alpha\in\cA_i}\mathcal{L}^{\alpha}_i(x^\eps,u_{m_0}(x^\eps),p_x,X)\ge 0.$$
This is the classical case (see \cite{ishii1991monotone}). In particular, by using the estimate
\begin{align*}
&\sup_{\alpha\in\cA_i}\mathcal{L}^{\alpha}_i(x^\eps,u_{m_0}(x^\eps),p_x,X)-\sup_{\alpha\in\cA_i}\mathcal{L}^{\alpha}_i(x^\eps,v(y^\eps),p_x,X)\\
&\ge \lambda_0((u_{m_0})_i(x^\eps)-v_i(y^\eps))=\lambda_0\bigg(M_\eps+\f{|x^\eps-y^\eps|^2}{2\eps}\bigg)
\end{align*}
and letting $\eps\to 0$, we can deduce that $M\le 0$, which is a contradiction. 
\end{proof}

\begin{proof}[Proof of Proposition \ref{prop:iter_conv}]
We start with several important properties of  the solution operator $Q: [C^0_1(\R^d)]^M\to [C^0_1(\R^d)]^M$  to \eqref{eq:iter_n}. That is, for any given $u$, $Qu$ solves the system of variational inequalities of the form \eqref{eq:iter_n}, where the obstacle $\cM_i u^{n-1}$ is replaced by $\cM_i u$. Then the comparison principle of \eqref{eq:iter_n}  and Lemma \ref{Lemma:cM} (2) imply that $Q$ is monotone: $Qu\ge Qv$ if $u\ge v$. Moreover, one can show $Q$ is concave. In fact, for any given $u,v \in  [C^0_1(\R^d)]^M$ and $\lambda\in [0,1]$, we can deduce from Lemma \ref{Lemma:cM} (1)  that for all $i\in \cI$, 
\bb\l{eq:M_convex}
(1-\lambda)(Qu)_i+\lambda (Qv)_i-\cM_i[(1-\lambda)u+\lambda v]\le (1-\lambda)((Qu)_i-\cM_iu)+\lambda((Qv)_i-\cM_iv).
\ee
Moreover,  since  the HJB equation \eqref{eq:iter_0} is convex in $u$, $Du$ and $D^2u$, by applying 
\cite[Lemma A.3]{barles2002} (note the weakly coupled term $\sum_{j\in \cI^{-i}}d^\a_{ij}u_j$ is linear in $u_j$, $j\in \cI$), we see $(1-\lambda)Qu+\lambda Qv$ is a subsolution to \eqref{eq:iter_n} with an obstacle $\cM_i[(1-\lambda)u+\lambda v]$, and  consequently conclude  the concavity of the operator $Q$ from  the comparison principle of \eqref{eq:iter_n}.

Now let $C$ be a sufficiently large constant such that  $w=(w_i)_{i\in\cI}$ with $w_i=-C$ for all $i\in\cI$ is a strict subsolution to \eqref{eq:qvi_impulse}, that is, $F_i(x,w,Dw_i,D^2w_i)\le -\kappa_0$ for all $i\in \cI$.
We proceed to establish a  contractive property of the iterates $(u^n)_{n\in\N}$, where
$u^n$ is a viscosity solution to \eqref{eq:iter_n} for each $n$.
By using the monotonicity and concavity of the operator $Q$, we can show that
if $u^{n-1}-  u^n \le \lambda (u^{n-1}-w)$ for some  $\lambda\in  [0,1]$ and $n\in \N$, then it holds for any constants $C\ge |(u^n)^+|_0+|w|_0$ and $0< \mu \le \min(1,\kappa_0/C)$ that $u^{n}-  u^{n+1} \le \lambda (1-\mu)(u^{n}-w)$  (cf.~\cite[Lemma~3.3]{reisinger2018qvi}).
Since $w\le u^n\le u^0$ for all $n$ and $w$ is  bounded, 
there exists a constant $\mu\in (0,1]$ such that $0\le u^{n-1}-u^n\le (1-\mu)^{n-1}(u^0-w)$ for all $n\ge 0$.
Consequently we can show $(u^{n})_{n\ge 0}$ converges uniformly to some continuous function $u$, which is the unique viscosity solution to \eqref{eq:qvi_impulse}. Then the contractive property enables us to conclude the desired error estimate.
\end{proof}

\begin{proof}[Proof of Lemma \ref{lemma:Qrho_lip}]
 For $\delta,\gamma>0$, we define for all $x,y\in \R^d$ that
\begin{align*}
\Phi_i(x,y)=(Q^\rho u)_i(x)-(Q^\rho v)_i(y)-\phi(x,y), \q \phi(x,y)=\delta|x-y|^2+\gamma |x|^2,
\end{align*}
and let  $\Phi_i(\bar{x},{y})=m_{\delta,\gamma}\coloneqq \sup_{i,x,y}\Phi_i(x,y)$ for some $(\bar{x},\bar{y})\in \R^{2d}$ and $i\in \cI$, where we omit the dependence on $\delta,\gamma$ for notational simplicity. Since $\cI$ is a finite set, we shall assume without loss of generality that the index $i$ is independent of $ \delta,\gamma$.
Then for any $\theta>1$,  we deduce from the maximum principle  \cite[Theorem 3.2]{crandall1992} that for any $\theta>1$, we have
\begin{align*}
\sup_{\alpha\in\cA_i}&\mathcal{L}^{\alpha}_i(\bar{x},(Q^\rho u)(\bar{x}),p_x,X)-\sup_{\alpha\in\cA_i}\mathcal{L}^{\alpha}_i(\bar{y},(Q^\rho v)(\bar{y}),-p_y,-Y)\\
&+\rho((Q^\rho u)_i-\cM_i u)^+(\bar{x})-\rho((Q^\rho v)_i-\cM_i v)^+(\bar{y})\le 0,
\end{align*}
where $(p_x,p_y)=(D_x\phi(\bar{x},\bar{y}), D_y\phi(\bar{x},\bar{y}))$, and 
$\begin{pmatrix} X &0 \\ 0& Y  \end{pmatrix}\le \theta D^2\phi (\bar{x},\bar{y})$.

We now discuss two cases. Suppose $((Q^\rho u)_i-\cM_iu)^+(\bar{x})-((Q^\rho v)_i-\cM_iv)^+(\bar{y})< 0$, then 
we have $((Q^\rho u)_i-\cM_iu)(\bar{x})\le ((Q^\rho v)_i-\cM_iv)(\bar{y})$, and consequently
\begin{align*}
(Q^\rho u)_i(\bar{x})-(Q^\rho v)_i(\bar{y})&\le (\cM_i u)(\bar{x})-(\cM_i v)(\bar{x})+(\cM_i v)(\bar{x})-(\cM_i v)(\bar{y})\\
&\le |(u_i-v_i)^+|_0+[\cM_iv]_1|\bar{x}-\bar{y}|,
\end{align*}
where we  used the definition \eqref{eq:M_impulse} of $\cM_i$.
This implies that
$$
m_{\delta,\gamma}\le |(u_i-v_i)^+|_0+[\cM_iv]_1|\bar{x}-\bar{y}|-\delta|\bar{x}-\bar{y}|^2\le  |(u_i-v_i)^+|_0+[\cM_iv]_1^2/(4\delta).
$$
Then, by passing $\gamma\to 0$, we deduce for any $x,y \in\R^d$ and $\delta>0$ that,
$$
(Q^\rho u)_i(x)-(Q^\rho v)_i(y)\le |(u_i-v_i)^+|_0+[\cM_iv]_1^2/(4\delta)+\delta |x-y|^2,
$$
which, along with the assumption $[\cM_iv]_1\le [v]_1+C$, leads to the desired conclusion by minimizing over $\delta>0$ and then setting $x=y$. 

On the other hand, if $((Q^\rho u)_i-u_j-k_{ij})^+(\bar{x})-((Q^\rho v)_i-v_j-k_{ij})^+(\bar{y})\ge 0$, then the classical results for weakly coupled system  gives us that $(Q^\rho u)_i\le (Q^\rho v)_i$ (see e.g.~\cite{ishii1991monotone}). 
\end{proof}

\begin{proof}[Proof of Lemma \ref{lemma:min}]
Note that for any given $\a>0$, $\mu\in (0,1)$ and $\gamma\in \N$, we have $(\phi^\a)'=\a \gamma x^{\gamma-1}+\mu^x\log\mu$, which is increasing on $(0,\infty)$. 
Suppose that $\a$ is sufficiently small such that $\a\gamma<-\log \mu$, then 
we can show $(\phi^\a)'(n^\a)\ge 0$, with the  natural number $n^\a$ defined as:
$$n^\a\coloneqq\left \lceil{ \f{\log(-\a \gamma/\log(\mu))}{\log \mu}}\right \rceil\le \f{\log(-\a \gamma/\log\mu)}{\log \mu}+1.$$
Consequently, $\phi^\a$ is increasing on $(n^\a,\infty)$, which leads to the estimate that for all small enough $\a$,
$$
m^\a \le \phi^\a(n^\a)\le \a \bigg(\f{\log(-\a \gamma/\log\mu)}{\log \mu}+1\bigg)^\gamma+\mu\f{-\a \gamma}{\log\mu}\le C\a(-\log \a)^\gamma,
$$
where  the constant $C$ depends only on $\gamma$ and $\mu$.
\end{proof}

%
%

\end{document}